\documentclass{fancypaper-cm}
\usepackage[extdef=true]{delimset}
\DeclareMathDelimiterSet{\scal}[2]{
  \selectdelim[l]< {#1} \mathpunct{}\selectdelim[p]| {#2} \selectdelim[r]>
}

\newcommand{\menge}[2]{
  \bigl\{{#1}\mid{#2}\bigr\}
}
\DeclareMathDelimiterSet{\Menge}[2]{
  \selectdelim[l]\{{#1}\selectdelim[m]|{#2}\selectdelim[r]\}
}

\newcommand*\Cdot{{\mkern 1.6mu\cdot\mkern 1.6mu}}

\newcommand{\NN}{\mathbb{N}}
\newcommand{\RR}{\mathbb{R}}

\newcommand{\HH}{\mathcal{H}}
\newcommand{\GG}{\mathcal{G}}

\newcommand{\UU}{\mathcal{U}}
\newcommand{\VV}{\mathcal{V}}
\newcommand{\WW}{\mathcal{W}}
\newcommand{\ZZ}{\mathcal{Z}}
\newcommand{\HHH}{\boldsymbol{\mathcal{H}}}
\newcommand{\VVV}{\boldsymbol{\mathcal{V}}}

\newcommand{\HS}{\mathsf{H}}
\newcommand{\EE}{\mathsf{E}}

\newcommand{\pinf}{{+}\infty}
\newcommand{\minf}{{-}\infty}
\newcommand{\emp}{\varnothing}
\newcommand{\exi}{\exists\,}
\renewcommand{\leq}{\leqslant}
\renewcommand{\geq}{\geqslant}

\newcommand{\zeroun}{\intv[o]{0}{1}}
\newcommand{\rzeroun}{\intv[l]{0}{1}}

\newcommand{\RX}{\intv[l]{\minf}{\pinf}}
\newcommand{\RP}{\RR_+}
\newcommand{\RPP}{\intv[o]{0}{\pinf}}

\DeclareMathOperator{\ran}{ran}

\DeclareMathOperator{\rem}{rem}
\DeclareMathOperator{\card}{card}

\newcommand{\Sum}{\displaystyle\sum}
\newcommand{\Id}{\mathrm{Id}}

\DeclareMathOperator{\dom}{dom}

\DeclareMathOperator{\proj}{proj}
\DeclareMathOperator{\prox}{prox}
\DeclareMathOperator{\reli}{ri}

\newcommand{\minimize}[2]{
  \underset{\substack{{#1}}}{\operatorname{minimize}}\;\;#2
}

\DeclareMathOperator{\cvar}{CVaR}

\crefname{model}{Model}{Models}
\theoremstyle{definition}
\newtheorem{model}[theorem]{Model}

\usepackage{tikz}
\usetikzlibrary{arrows.meta,positioning}		

\usepackage{booktabs}
\usepackage{multirow}

\begin{document}


\author{Minh N. B\`{u}i}
\affil{%
  Chongqing Normal University
  \affilcr
  School of Mathematical Sciences
  \affilcr
  Shapingba District, Chongqing 401331, P.~R.~China
  \affilcr
  \email{mnb@cqnu.edu.cn}
}

\title{A Block-Activated Decomposition Algorithm for\\
  Multi-Stage Stochastic Variational Inequalities\thanks{%
  This work was partially supported by the National Natural Science
  Foundation of China grant number 12671353.}}

\date{~}

\thispagestyle{plain.scrheadings}

\maketitle

\begin{abstract}
We develop a block-activated decomposition algorithm for
multi-stage stochastic variational inequalities with
nonanticipativity constraints. The algorithm features two computational
novelties: (i) At each iteration, it activates only a
user-chosen subset of scenarios. (ii) For each activated scenario,
it employs the resolvent of the cost operator and the projector
onto the constraint set separately. These reduce computational load
and enhance tractability, in contrast with existing approaches,
which often require evaluating the resolvent of the sum of the cost
operator and the normal cone operator of the constraint set.
As an application, we demonstrate that in risk-averse stochastic
programming with conditional value-at-risk objective functions,
our method requires only projecting onto constraint sets,
together with solving univariate equations involving the proximity
operators of the cost functions, thereby avoiding solving
high-dimensional constrained subproblems as required by existing
methods.
\end{abstract}


\newpage

\section{Introduction}
\label{sec:1}

Motivated by the multi-stage stochastic variational inequality
model developed by Rockafellar and Wets in \cite{Rockafellar17}, we
consider in this paper the multi-stage stochastic equilibrium
\cref{prob:1}, which will be formulated within the following
multi-stage decision-making model of \cite{Rockafellar91}.

\begin{model}
\label{m:1}
We consider an $N$-stage decision-making process
with \emph{nonanticipativity}, where $N\in\NN\smallsetminus\set{0}$.
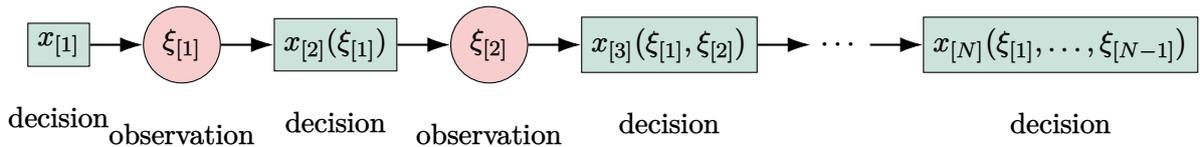
\begin{center}
\begin{tikzpicture}[%
node distance=3.9ex,
every node/.style={font=\normalfont},
decision/.style={draw, rectangle, fill=dgreen!20},
observation/.style={draw, circle, fill=dred!20},
arrow/.style={thick, -{Stealth[length=3.0mm, width=2mm]}}
]

\node[decision] (x1) {$x_{[1]}$};
\node[observation, right=of x1] (xi1) {$\xi_{[1]}$};
\node[decision, right=of xi1] (x2) {$x_{[2]}(\xi_{[1]})$};
\node[observation, right=of x2] (xi2) {$\xi_{[2]}$};
\node[decision, right=of xi2] (x3) {$x_{[3]}\brk!{\xi_{[1]},\xi_{[2]}}$};
\node[right=of x3] (dots) {$\cdots$};
\node[decision, right=of dots] (xN)
{$x_{[N]}\brk!{\xi_{[1]},\ldots,\xi_{[N-1]}}$};

\draw[arrow] (x1) -- (xi1);
\draw[arrow] (xi1) -- (x2);
\draw[arrow] (x2) -- (xi2);
\draw[arrow] (xi2) -- (x3);
\draw[arrow] (x3) -- (dots);
\draw[arrow] (dots) -- (xN);

\node[below=0.4cm of x1] {\small decision};
\node[below=0.4cm of xi1] {\small observation};
\node[below=0.4cm of x2] {\small decision};
\node[below=0.4cm of xi2] {\small observation};
\node[below=0.4cm of x3] {\small decision};
\node[below=0.4cm of xN] {\small decision};
\end{tikzpicture}
\end{center}
Here, the set $\Xi$ of all possible scenarios for the process
is assumed to be finite, with known probability $\pi(\xi)\in\rzeroun$
for each scenario $\xi=(\xi_{[1]},\ldots,\xi_{[N]})\in\Xi$,
where $\xi_{[k]}$ designates the information revealed
at stage $k$ after the $k$th-stage decision has been made but before
the $(k+1)$th-stage decision is made\footnote{For clarity, we use
square brackets in the subscript $[k]$ to distinguish stage indices
from sequence indices.}.
Given a scenario $\xi\in\Xi$, $x_{[k]}(\xi)\in\RR^{d_k}$
designates a decision vector at the $k$th stage and
\begin{equation}
\label{e:rdxk}
x(\xi)
=\brk!{x_{[1]}(\xi),\ldots,x_{[N]}(\xi)}
\in\RR^{d_1}\times\cdots\times\RR^{d_N}=\RR^d,
\quad\text{where}\,\,d=d_1+\cdots+d_N,
\end{equation}
designates a full decision vector, which thus defines mappings
$x_{[k]}\colon\Xi\to\RR^{d_k}$ and $x\colon\Xi\to\RR^d$.
We equip the set
\begin{equation}
\HH=\set!{x\colon\Xi\to\RR^d\colon
\xi\mapsto\brk!{x_{[1]}(\xi),\ldots,x_{[N]}(\xi)}}
\end{equation}
with scenario-wise addition and scalar multiplication,
together with the expectation scalar product
\begin{equation}
(\forall x\in\HH)(\forall y\in\HH)\quad
\scal{x}{y}_{\HH}
=\EE\brk!{\scal{x(\Cdot)}{y(\Cdot)}_2}
=\sum_{\xi\in\Xi}\pi(\xi)\scal{x(\xi)}{y(\xi)}_2,
\end{equation}
where $\scal{\Cdot}{\Cdot}_2$ is the $\ell^2$ scalar product
on $\RR^d$.
We next formulate the nonanticipativity constraint, which ensures
that each $k$th-stage decision mapping $x_{[k]}$ depends only on
information revealed before the $k$th-stage decision is made and, in
particular, that the first-stage decision is scenario-independent.
To do so, the following equivalence relation will be useful.
At stage $k\in\set{1,\ldots,N}$, we say that two scenarios
$\xi=(\xi_{[1]},\ldots,\xi_{[N]})\in\Xi$
and
$\eta=(\eta_{[1]},\ldots,\eta_{[N]})\in\Xi$
are \emph{information-equivalent},
written $\xi\sim_k\eta$, if they provide the same past information
for making the $k$th-stage decision, that is,
\begin{equation}
\xi\sim_k\eta
\quad\Longleftrightarrow\quad
\brk!{\xi_{[1]},\ldots,\xi_{[k-1]}}
=\brk!{\eta_{[1]},\ldots,\eta_{[k-1]}}.
\end{equation}
As is customary,
the collection of equivalence classes of
$\sim_k$ is denoted by $\Xi/\sim_k$. Note that
$\Xi/\sim_1=\set{\Xi}$. Thus, the nonanticipativity constraint is
modeled by the vector subspace
\begin{equation}
\VV=\menge{x\in\HH}{\brk!{\forall k\in\set{1,\ldots,N}}
(\forall S\in\Xi/\sim_k)\,\,
\text{$x_{[k]}$ is constant on $S$}}
\end{equation}
of $\HH$.
\end{model}

We now state our main problem and refer the reader to
\cref{sec:2} for notation and background on monotone operator
theory.

\begin{problem}[multi-stage stochastic equilibrium]
\label{prob:1}
Consider \cref{m:1}.
For each scenario $\xi\in\Xi$,
let $A(\xi,\Cdot)\colon\RR^d\to 2^{\RR^d}$ be maximally monotone
(with respect to the $\ell^2$ scalar product) and
let $\emp\neq C(\xi)\subset\RR^d$ be closed and convex.
The task is to solve the equilibrium problem
\begin{equation}
\label{e:1}
\text{find $x\in\VV$ and $v^*\in\VV^{\perp}$ such that}\,\,
(\forall\xi\in\Xi)\,\,
{-}v^*(\xi)\in A\brk!{\xi,x(\xi)}+N_{C(\xi)}x(\xi),
\end{equation}
under the assumption that a solution exists.
\end{problem}

In \cref{prob:1}, if each $A(\xi,\Cdot)$ is single-valued,
then \cref{e:1} reduces to the multi-stage stochastic variational
inequality proposed by Rockafellar and Wets in
\cite[Definition~3.1]{Rockafellar17}, that is,
\begin{equation}
\label{e:rw17}
\text{find $x\in\VV$ and $v^*\in\VV^{\perp}$ such that}\,\,
(\forall\xi\in\Xi)\,\,
{-}A\brk!{\xi,x(\xi)}-v^*(\xi)\in N_{C(\xi)}x(\xi).
\end{equation}
In \cite{Rockafellar19} Rockafellar and Sun proposed a decomposition
method for solving \cref{e:rw17} by extending the classical
progressive hedging algorithm of \cite{Rockafellar91} to the
context of monotone operators.
The core idea is that these methods are instantiations of Spingarn's
method of partial inverses (see, e.g.,
\cite[Section~5.4.4]{Combettes24}).
In essence, under the assumption that
\begin{equation}
(\forall\xi\in\Xi)\quad
\text{$A(\xi,\Cdot)$ is single-valued
and $A(\xi,\Cdot)+N_{C(\xi)}$ is maximally monotone},
\end{equation}
and given a starting point $(x_0,v_0^*)\in\VV\times\VV^{\perp}$
and a parameter $\gamma\in\RPP$, the progressive hedging algorithm
\cite[Eq.~(2.5)]{Rockafellar19} reads
\begin{equation}
\label{e:PH}
\begin{array}{l}
\text{for $n=0,1,\ldots$}\\
\left\lfloor
\begin{array}{l}
\text{for each scenario $\xi\in\Xi$}\\
\left\lfloor
\begin{array}{l}
a_n(\xi)=J_{\gamma\brk{A(\xi,\Cdot)+N_{C(\xi)}}}\brk!{
x_n(\xi)-\gamma v_n^*(\xi)}
\end{array}
\right.\\
x_{n+1}=\proj_{\VV}a_n\\
v_{n+1}^*=v_n^*+\gamma^{-1}\proj_{\VV^{\perp}}a_n.
\end{array}
\right.
\end{array}
\end{equation}
Despite \cref{e:PH} being a parallel method, its efficiency hinges
on two critical factors:
\begin{enumerate}[label={\normalfont\bfseries F{\arabic*}}]
\item\label{cf:1}
The computational ability to process \emph{all} subproblems
\begin{equation}
\label{e:a24j}
\begin{array}{l}
\text{for each scenario $\xi\in\Xi$}\\
\left\lfloor
\begin{array}{l}
a_n(\xi)=J_{\gamma\brk{A(\xi,\Cdot)+N_{C(\xi)}}}\brk!{
x_n(\xi)-\gamma v_n^*(\xi)}
\end{array}
\right.
\end{array}
\end{equation}
at \emph{every} iteration---with the implicit assumption that
\begin{equation}
\label{e:impl}
\text{these subproblems can be handled effectively.}
\end{equation}
The step \cref{e:a24j} becomes computationally demanding when the
number of scenarios is substantial.
The implicit assumption \cref{e:impl} naturally leads to the next
critical factor.
\item\label{cf:2}
The tractability of evaluating the resolvents
\begin{equation}
\brk!{J_{\gamma\brk{A(\xi,\Cdot)+N_{C(\xi)}}}}_{\xi\in\Xi}
\end{equation}
in \cref{e:a24j}. In general, such resolvents
cannot be explicitly expressed via those of the constituent
operators, a difficulty that has long been recognized.
Thus, each subproblem in \cref{e:PH} requires
an iterative solution method to evaluate
$J_{\gamma\brk{A(\xi,\Cdot)+N_{C(\xi)}}}$,
with the additional challenge that inexact solutions
to these subproblems can compromise the overall convergence of the
algorithm.
This is particularly acute in applications
where the cost operators $A(\xi,\Cdot)$ have complex structure,
such as those arising from reformulating multi-stage stochastic
programming with risk measures as a classical stochastic
programming problem \cite{Rockafellar18}
(see \cref{sec:3} for further detail).
\end{enumerate}
We emphasize that, to the best of our knowledge:
\begin{itemize}
\item
The computational load issue \cref{cf:1} has been addressed
\emph{only} in the minimization setting.
\item
The tractability issue \cref{cf:2} does not appear to have been
addressed, even in the minimization setting;
\end{itemize}
see \cref{r:1} for further discussion.

To circumvent these drawbacks of the progressive hedging algorithm
\cref{e:PH}, it is therefore the goal of this work to propose a
novel operator splitting technique that results in a
block-activated decomposition algorithm with the following features
for solving \cref{prob:1}:
\begin{enumerate}[label={\normalfont\bfseries G\arabic*}]
\item
\label{G1}
To reduce the computational load, only a user-chosen subset $\Xi_n$
of scenarios is activated at every iteration.
We ask only that, within any $T+1$ consecutive iterations, where
$T$ is user-chosen and iteration-independent,
each scenario must be activated at least once.
\item
\label{G2}
For each activated scenario $\xi\in\Xi_n$,
our algorithm does not evaluate
$J_{\gamma\brk{A(\xi,\Cdot)+N_{C(\xi)}}}$, but rather
the resolvent $J_{\gamma A(\xi,\Cdot)}$ and the projector
$\proj_{C(\xi)}$ separately, which are more tractable.
\item
\label{G3}
It generates a sequence $(x_n)_{n\in\NN}$ of implementable
policies, that is, $\set{x_n}_{n\in\NN}\subset\VV$,
together with a sequence $(v_n^*)_{n\in\NN}$ in $\VV^{\perp}$,
such that $(x_n,v_n^*)_{n\in\NN}$
converges to a solution to \cref{prob:1}.
\end{enumerate}
In the current literature, there does not appear to be a technique
that can \emph{simultaneously} achieve \cref{G1,G2,G3} for the
general \cref{prob:1}.

The paper is organized as follows.
We discuss related work and our approach in \cref{sec:1bis}.
In \cref{sec:2} we present our algorithm, establish its convergence
properties, and discuss its features in detail.
\cref{sec:3} is devoted to a novel application:
By employing Rockafellar's reformulation technique
\cite{Rockafellar18},
we illustrate how our method leads to
a block-activated algorithm for multi-stage stochastic programming
with conditional value-at-risk objective functions
that, for each activated scenario,
\begin{itemize}
\item
separately employs the projector onto the constraint set,
\item
solves only a univariate equation involving the proximity operator
of the cost function, rather than challenging constrained
subproblems as required by existing methods.
\end{itemize}
Finally, to illustrate the computational benefits of features
\cref{G1,G2}, we present in \cref{sec:5} numerical experiments
comparing the proposed algorithm, under both full and partial
scenario activation, with the progressive hedging
algorithm \cref{e:PH} of Rockafellar and Sun.

\section{Related work and our approach}
\label{sec:1bis}

To facilitate our discussion, let us provide a special case of
\cref{prob:1}.

\begin{example}
\label{ex:1}
In \cref{prob:1},
suppose that each $A(\xi,\Cdot)$ is the subdifferential of a
lower semicontinuous convex function
$f(\xi,\Cdot)\colon\RR^d\to\RX$.
A standard convex analytic argument then shows that,
if $(\overline{x},\overline{v}^*)\in\HH\times\HH$ solves \cref{e:1},
then $\overline{x}$ solves the constrained multi-stage stochastic
programming with \emph{risk-neutral cost}
\begin{equation}
\label{e:19rn}
\minimize{\substack{x\in\VV\\
(\forall\xi\in\Xi)\,\,x(\xi)\in C(\xi)}}{
\EE\brk!{f\brk!{\Cdot,x(\Cdot)}}},
\end{equation}
where the expectation is taken with respect to the
probability space $\Xi$.
Furthermore, under a suitable constraint qualification condition,
one can show that $\overline{x}\in\HH$ solves
\cref{e:19rn} if and only if there exists
$\overline{v}^*\in\VV^{\perp}$ such that
$(\overline{x},\overline{v}^*)$ solves \cref{e:1}.
\end{example}

\begin{remark}
\label{r:1}
In the context of \cref{ex:1}:
\begin{enumerate}
\item
The computational load drawback
\cref{cf:1} of the classical progressive hedging algorithm
\cite{Rockafellar91} motivated the recent works
\cite{Bareilles20,Eckstein25} to develop block-activated methods
for solving the special case \cref{e:19rn} of \cref{prob:1}
\emph{under the implicit assumption \cref{e:impl}}.
There does not seem to exist a path to achieve
\cref{G2} from the operator splitting techniques of
\cite{Bareilles20,Eckstein25}.
\item
In solution frameworks for multi-stage stochastic optimization
\cite{Bareilles20,Eckstein25,Pennanen06,Rockafellar91},
the tractability issue \cref{cf:2} is often obscured by
incorporating the constraint sets $C(\xi)$ into the cost functions
$f(\xi,\Cdot)$, that is, by rewriting \cref{e:19rn} as
\begin{equation}
\minimize{x\in\VV}{\EE\brk!{g\brk!{\Cdot,x(\Cdot)}}},
\quad
\text{where}\,\,
(\forall\xi\in\Xi)\,\,
g(\xi,\Cdot)=f(\xi,\Cdot)+\iota_{C(\xi)}.
\end{equation}
\end{enumerate}
\end{remark}

Let us now sketch our path to achieve the goals
\cref{G1,G2,G3} set forth in the Introduction.
The first step is to identify a suitable class of abstract operator
splitting methods. To achieve \cref{G1},
one is tempted to consider stochastic operator
splitting methods (including the stochastic version of the
Douglas--Rachford algorithm;
see, e.g., \cite{Bareilles20,Combettes25,Combettes15}
and the references cited therein).
These stochastic schemes, however, rely on random activations of
scenarios and therefore do not allow ``fine-grained'' control
over which scenarios (and thus, cost operators) are being used.
We thus chose to adopt the \emph{projective splitting
algorithm} of \cite{Combettes18}, which features
\emph{deterministic} block-activation, that is, users have
complete control of per-iteration computational load;
see \cite{Icassp22,Combettes24,Combettes18,Eckstein25} for further
discussions on and numerical illustrations of this recent class of
splitting algorithms.

Our choice was also motivated by the empirical evidence
in \cite{Eckstein25} for solving \cref{ex:1}, as we now describe.

\begin{remark}[projective hedging algorithm]
\label{r:6}
Based on the projective splitting algorithm of \cite{Eckstein17}
(a simplified version of \cite{Combettes18}),
the recent work \cite{Eckstein25} developed an
asynchronous deterministic block-activated method, called
\emph{projective hedging algorithm}, to solve \cref{e:19rn}
under the following assumptions:
\begin{itemize}
\item (Compactness of domain)
For each scenario $\xi\in\Xi$,
$\dom(f(\xi,\Cdot)+\iota_{C(\xi)})$ is compact;
see \cite[Assumption~1]{Eckstein25}.
\item (Reliability of non-strongly-convex subproblem solvers)
For each scenario $\xi\in\Xi$,
non-strongly-convex subproblems of the form
\begin{equation}
\minimize{\mathsf{x}\in C(\xi)}{
f(\xi,\mathsf{x})
+\scal!{\mathsf{w}}{
\brk!{\mathsf{x}_{[1]},\ldots,\mathsf{x}_{[N-1]}}}_2
+\frac{1}{2\gamma}\norm!{\brk!{\mathsf{x}_{[1]},\ldots,\mathsf{x}_{[N-1]}}-\mathsf{z}}_2^2
},
\end{equation}
can be solved reliably,
where $\mathsf{w}$ and $\mathsf{z}$
are given vectors in $\RR^{d_1}\times\cdots\times\RR^{d_{N-1}}$
and for a vector
$\mathsf{x}\in\RR^d=\RR^{d_1}\times\cdots\times\RR^{d_N}$,
$\mathsf{x}_{[k]}\in\RR^{d_k}$ denotes the $k$th block of
$\mathsf{x}$ (see also \cref{e:rdxk});
see line~7 of \cite[Algorithm~2]{Eckstein25}.
\end{itemize}
Large-scale numerical experiments were carried out
in \cite{Eckstein25} to demonstrate the advantages of the
projective hedging algorithm---particularly,
of deterministic block-activation---over
the classical progressive hedging algorithm \cite{Rockafellar91}.
More precisely, in parallel computational experiments
(48 to 6,000 processor cores),
\cite{Eckstein25} demonstrated a large-margin superiority of the
projective hedging algorithm in test problems
with two or five stages, and the number of scenarios ranging
from 20,000 to 1,000,000.
However, the approach of \cite{Eckstein25} was developed
specifically for the minimization setting and it is not clear how
it can be extended to solve \cref{prob:1} and to achieve
\cref{G2}.
\end{remark}

Having identified a suitable abstract operator splitting framework,
it remains to devise a technique for achieving \cref{G2}.
``Na\"{i}ve applications'' of this framework to
\cref{prob:1} often force two operators (namely,
$A(\xi,\Cdot)$ and $N_{C(\xi)}$) to be grouped together as a sum,
thereby preventing \cref{G2}. To overcome this difficulty, we
propose a novel reformulation of \cref{prob:1}, to which we apply a
less obvious form of \cite[Algorithm~12]{Combettes18}.

Finally, we note that our work here can be viewed as an extension
of the projective hedging algorithm to the general \cref{prob:1},
with the additional novelty of lifting its two restrictive
assumptions. In fact, in the minimization setting, our method will
be shown to reduce to an algorithm that strongly resembles the
projective hedging algorithm when one wishes to solve
the subproblems of evaluating
$J_{\gamma\brk{A(\xi,\Cdot)+N_{C(\xi)}}}$ at each iteration; see
\cref{r:5} for a detailed discussion.

\section{Solving \texorpdfstring{\cref{prob:1}}{the proposed problem}}
\label{sec:2}

We start by recalling several standard terminologies from
monotone operator theory.
Let $\HH$ be a Euclidean space with scalar product
$\scal{\Cdot}{\Cdot}$ and associated norm $\norm{\Cdot}$.
A set-valued operator $M\colon\HH\to 2^{\HH}$ is monotone if
\begin{equation}
  (\forall x\in\HH)(\forall y\in\HH)
  (\forall x^*\in Mx)(\forall y^*\in My)\quad
  \scal{x-y}{x^*-y^*}\geq 0,
\end{equation}
and maximally monotone if it is monotone and
\begin{multline}
  \brk!{\forall (x,x^*)\in\HH\times\HH}\\
  \brk[s]!{\;
    (\forall y\in\HH)(\forall y^*\in My)\,\,
    \scal{x-y}{x^*-y^*}\geq 0\;}
  \quad\Longrightarrow\quad
  x^*\in Mx.
\end{multline}
The resolvent of a maximally monotone operator
$M\colon\HH\to 2^{\HH}$ is
\begin{equation}
  J_M=(\Id+M)^{-1}\colon\HH\to\HH.
\end{equation}
Next, let $f\colon\HH\to\RX$ be proper, lower semicontinuous, and
convex. The subdifferential of $f$ is the maximally monotone operator
\begin{equation}
  \partial f\colon\HH\to 2^{\HH}\colon
  x\mapsto
  \menge{x^*\in\HH}{(\forall y\in\HH)\,\,
    \scal{y-x}{x^*}+f(x)\leq f(y)},
\end{equation}
the proximity operator of $f$ is the mapping
$\prox_f\colon\HH\to\HH$ which maps every $x\in\HH$ to the unique
minimizer of the strongly convex function
$y\mapsto f(y)+\norm{x-y}^2/2$, and we have
\begin{equation}
  \prox_f=J_{\partial f}.
\end{equation}
In particular, given a nonempty closed convex subset $C$ of $\HH$,
the normal cone operator of $C$ is
$N_C=\partial\iota_C$ and
the projector onto $C$ is $\proj_C=\prox_{\iota_C}$,
where
\begin{equation}
  \iota_C\colon\HH\to\RX\colon x\mapsto
  \begin{cases}
    0,&\text{if}\,\,x\in C;\\
    \pinf,&\text{otherwise}
  \end{cases}
\end{equation}
is the indicator function of $C$.
For notations which are not explicitly defined here
and for background and complements on monotone operator theory and
convex analysis, we refer the reader to \cite{Livre1,Combettes24}.

The proposed algorithm and its convergence properties are
established in the following theorem.

\begin{theorem}
\label{t:1}
Consider the setting of \cref{prob:1}. Let $(\Xi_n)_{n\in\NN}$ be a
sequence of nonempty subsets of $\Xi$ such that
\begin{equation}
\label{e:3kfs}
\Xi_0=\Xi
\quad\text{and}\quad
(\exi T\in\NN)(\forall n\in\NN)\,\,
\bigcup_{j=n}^{n+T}\Xi_j=\Xi,
\end{equation}
let $\varepsilon\in\zeroun$, and for every $\xi\in\Xi$,
let $(\gamma_{\xi,n})_{n\in\NN}$ and $(\mu_{\xi,n})_{n\in\NN}$ be
sequences in $\intv{\varepsilon}{1/\varepsilon}$.
Additionally, let $(U(\xi))_{\xi\in\Xi}$ be a family of vector
subspaces of $\RR^d$ such that
\begin{equation}
\label{e:x2qw}
(\forall\xi\in\Xi)\quad
\ran\brk!{\Id-\proj_{C(\xi)}}\subset U(\xi),
\end{equation}
let $x_0\in\VV$, let $v_0^*\in\VV^{\perp}$, and let $x_0^*\in\HH$ be
such that $(\forall\xi\in\Xi)$ $x_0^*(\xi)\in U(\xi)$. Iterate
\begin{equation}
\label{e:algo}
\begin{array}{l}
\text{for $n=0,1,\ldots$}\\
\left\lfloor
\begin{array}{l}
\text{for each scenario $\xi\in\Xi_n$}\\
\left\lfloor
\begin{array}{l}
l_n^*(\xi)=x_n^*(\xi)+v_n^*(\xi)
\smallskip\\
a_n(\xi)=
J_{\gamma_{\xi,n}A(\xi,\Cdot)}\brk!{x_n(\xi)
-\gamma_{\xi,n}l_n^*(\xi)}
\smallskip\\
a_n^*(\xi)=\gamma_{\xi,n}^{-1}\brk!{x_n(\xi)-a_n(\xi)}-l_n^*(\xi)
\smallskip\\
b_n(\xi)=\proj_{C(\xi)}\brk!{x_n(\xi)+\mu_{\xi,n}x_n^*(\xi)}
\smallskip\\
b_n^*(\xi)=x_n^*(\xi)+\mu_{\xi,n}^{-1}\brk!{x_n(\xi)-b_n(\xi)}
\smallskip\\
u_n(\xi)=\proj_{U(\xi)}\brk!{b_n(\xi)-a_n(\xi)}
\end{array}
\right.\\
\text{for each scenario $\xi\in\Xi\smallsetminus\Xi_n$}\\
\left\lfloor
\begin{array}{l}
a_n(\xi)=a_{n-1}(\xi);\,\,
a_n^*(\xi)=a_{n-1}^*(\xi)\\
b_n(\xi)=b_{n-1}(\xi);\,\,
b_n^*(\xi)=b_{n-1}^*(\xi)\\
u_n(\xi)=u_{n-1}(\xi)
\end{array}
\right.\\
t_n^*=\proj_{\VV}(a_n^*+b_n^*)\\
t_n={-}\proj_{\VV^{\perp}}a_n\\
\tau_n=\norm{t_n^*}_{\HH}^2+\norm{u_n}_{\HH}^2+\norm{t_n}_{\HH}^2\\
\text{if $\tau_n>0$}\\
\left\lfloor
\begin{array}{l}
\kappa_n=\scal{x_n}{t_n^*}_{\HH}-\scal{a_n}{a_n^*}_{\HH}
+\scal{u_n}{x_n^*}_{\HH}-\scal{b_n}{b_n^*}_{\HH}
+\scal{t_n}{v_n^*}_{\HH}\\
\lambda_n\in\intv{\varepsilon}{2-\varepsilon}\\
\theta_n=\lambda_n\max\set{\kappa_n,0}/\tau_n
\end{array}
\right.\\
\text{else}\\
\left\lfloor
\begin{array}{l}
\theta_n=0
\end{array}
\right.\\
x_{n+1}=x_n-\theta_nt_n^*;\;
x_{n+1}^*=x_n^*-\theta_nu_n;\;
v_{n+1}^*=v_n^*-\theta_nt_n.
\end{array}
\right.
\end{array}
\end{equation}
Then there exists a solution
$(\overline{x},\overline{v}^*)\in\VV\times\VV^{\perp}$ to
the multi-stage stochastic equilibrium problem \cref{e:1} such that
$x_n\to\overline{x}$ and
$v_n^*\to\overline{v}^*$ in $\HH$.
Furthermore, for every $n\in\NN$,
$x_n$ is an \emph{implementable policy}
(in the terminology of Rockafellar and Wets \cite{Rockafellar91}),
that is, $x_n\in\VV$ and $v_n^*\in\VV^{\perp}$.
\end{theorem}

\begin{proof}
We show that \cref{e:algo} is a realization of
\cite[Algorithm~12]{Combettes18} to the system of monotone
inclusions
\begin{multline}
\label{e:2}
\text{find $x\in\HH$ and $(x^*,v^*)\in\HH\times\HH$ such that}\\
\begin{cases}
(\forall\xi\in\Xi)\,\,
\begin{cases}
{-}x^*(\xi)-v^*(\xi)\in A\brk!{\xi,x(\xi)}\\
x^*(\xi)\in N_{C(\xi)}\brk!{x(\xi)}
\end{cases}
\\
v^*\in N_{\VV}x
\end{cases}
\end{multline}
with the choice
\begin{equation}
\mathsf{K}=\VV\times\UU\times\VV^{\perp},
\quad\text{where}\,\,
\UU=\menge{x\in\HH}{(\forall\xi\in\Xi)\,\,x(\xi)\in U(\xi)},
\end{equation}
and the conclusion will then follow from
\cite[Theorem~13]{Combettes18}.
To do so, we first put \cref{e:2} in the form of
\cite[Problem~1]{Combettes18}, to which
\cite[Algorithm~12]{Combettes18} is applicable.
Since $\Xi$ is a finite set, we assume, for presentation purposes,
that
\begin{equation}
\label{e:inde}
\text{$\Xi=\set{1,\ldots,m}$ for some
$m\in\NN\smallsetminus\set{0}$}.
\end{equation}
For every $\xi\in\Xi$, let $\HS_{\xi}$ be the Euclidean space
obtained by equipping the standard vector space $\RR^d$ with the
scalar product
\begin{equation}
\scal{\Cdot}{\Cdot}_{\HS_{\xi}}=\pi(\xi)\scal{\Cdot}{\Cdot}_2.
\end{equation}
To make the connection between \cref{e:2} and
\cite[Problem~1]{Combettes18} clear,
we identify $\HH$ with the Hilbert direct sum
$\bigoplus_{\xi\in\Xi}\HS_{\xi}$ and identify a mapping
$x\in\HH$ with the tuple
$(x(1),\ldots,x(m))\in\bigoplus_{\xi\in\Xi}\HS_{\xi}$
consisting of its values.
Now set
\begin{equation}
\label{e:kbkx}
K=\Xi\cup\set{m+1},
\quad
B_{m+1}=N_{\VV}\colon\HH\to 2^{\HH},
\quad\text{and}\quad
(\forall\xi\in\Xi)\,\,
B_{\xi}=N_{C(\xi)}\colon\HS_{\xi}\to 2^{\HS_{\xi}},
\end{equation}
and define a family $(L_{k,\xi})_{(k,\xi)\in K\times\Xi}$
of linear operators as follows:
\begin{equation}
(\forall\eta\in\Xi)(\forall\xi\in\Xi)\quad
L_{\eta,\xi}=
\begin{cases}
\Id\colon\HS_{\xi}\to\HS_{\xi},&\text{if}\,\,\eta=\xi;\\
0\colon\HS_{\xi}\to\HS_{\eta},
&\text{if}\,\,\eta\neq\xi,
\end{cases}
\end{equation}
and, for every $\xi\in\Xi$, $L_{m+1,\xi}\colon\HS_{\xi}\to\HH$
is defined via
\begin{equation}
(\forall\mathsf{x}\in\HS_{\xi})\quad
L_{m+1,\xi}\mathsf{x}\colon\Xi\to\RR^d\colon
\eta\mapsto
\begin{cases}
\mathsf{x},&\text{if}\,\,\eta=\xi;\\
\mathsf{0},&\text{if}\,\,\eta\neq\xi.
\end{cases}
\end{equation}
Then the operators $(B_k)_{k\in K}$ are maximally
monotone\footnote{For each scenario $\xi\in\Xi$,
the normal cone of $C(\xi)$ with respect to the $\ell^2$
scalar product on $\RR^d$ coincides with that with respect to
the scalar product $\scal{\Cdot}{\Cdot}_{\HS_{\xi}}$.}
\cite[Example~20.26]{Livre1}, and we have
\begin{equation}
\label{e:832y}
(\forall k\in K)(\forall x\in\HH)\quad
\sum_{\xi\in\Xi}L_{k,\xi}\brk!{x(\xi)}=
\begin{cases}
x(k),&\text{if}\,\,k\in\Xi;\\
x,&\text{if}\,\,k=m+1
\end{cases}
\end{equation}
and
\begin{equation}
\label{e:832z}
(\forall\xi\in\Xi)(\forall x^*\in\HH)(\forall v^*\in\HH)\quad
\Sum_{k\in\Xi}L_{k,\xi}^*\brk!{x^*(k)}+L_{m+1,\xi}^*v^*
=x^*(\xi)+v^*(\xi),
\end{equation}
where $L_{k,\xi}^*$ denotes the adjoint of $L_{k,\xi}$.
Therefore, using the fact that $N_{\VV}^{-1}=N_{\VV^{\perp}}$
\cite[Example~6.43]{Livre1}, we rewrite \cref{e:2} as
\begin{multline}
\label{e:2bis}
\text{find $\brk!{x(\xi)}_{\xi\in\Xi}\in\bigoplus_{\xi\in\Xi}\HS_{\xi}$
and
$\brk!{\brk!{x^*(\xi)}_{\xi\in\Xi},v^*}\in\brk3{\bigoplus_{\xi\in\Xi}\HS_{\xi}}\oplus\HH$
such that}\\
\begin{cases}
(\forall\xi\in\Xi)\,\,
{-}\Sum_{k\in\Xi}L_{k,\xi}^*\brk!{x^*(k)}-L_{m+1,\xi}^*v^*
\in A\brk!{\xi,x(\xi)}
\smallskip\\
(\forall k\in\Xi)\,\,
\Sum_{\xi\in\Xi}L_{k,\xi}\brk!{x(\xi)}
\in B_k^{-1}\brk!{x^*(k)}
\smallskip\\
\Sum_{\xi\in\Xi}L_{m+1,\xi}\brk!{x(\xi)}
\in B_{m+1}^{-1}v^*,
\end{cases}
\end{multline}
which is precisely \cite[Problem~1]{Combettes18} with the operators
$(A_i)_{i\in I}$ there being $(A(\xi,\Cdot))_{\xi\in\Xi}$.
(For clarity, we point out that, for each scenario $\xi\in\Xi$,
$A(\xi,\Cdot)$ is also maximally monotone with respect to
the scalar product $\scal{\Cdot}{\Cdot}_{\HS_{\xi}}$.)
Next, denote by $\boldsymbol{\ZZ}$ the set of solutions to
\cref{e:2}. Since
\begin{equation}
N_{\VV}\colon\HH\to 2^{\HH}\colon
x\mapsto
\begin{cases}
\VV^{\perp},&\text{if}\,\,x\in\VV;\\
\emp,&\text{otherwise},
\end{cases}
\end{equation}
we deduce that
\begin{equation}
\label{e:r46z}
\brk!{\forall(x,v^*)\in\HH\times\HH}\quad
\text{$(x,v^*)$ solves \cref{e:1}}
\quad\Longleftrightarrow\quad
(\exi x^*\in\HH)\,\,
\text{$(x,x^*,v^*)$ solves \cref{e:2}}.
\end{equation}
In turn, since \cref{e:1} has a solution,
$\boldsymbol{\ZZ}\neq\emp$.
At the same time, for every $(x,x^*,v^*)\in\boldsymbol{\ZZ}$
and every $\xi\in\Xi$, it results from \cref{e:2} and the identity
$\proj_{C(\xi)}=(\Id+N_{C(\xi)})^{-1}$ that
\begin{equation}
x(\xi)=\proj_{C(\xi)}\brk!{x(\xi)+x^*(\xi)},
\end{equation}
which leads to
\begin{equation}
x^*(\xi)
=\brk!{x(\xi)+x^*(\xi)}-\proj_{C(\xi)}\brk!{x(\xi)+x^*(\xi)}
\in\ran\brk!{\Id-\proj_{C(\xi)}}
\subset U(\xi),
\end{equation}
where the last inclusion is the condition \cref{e:x2qw}.
Hence, $\boldsymbol{\ZZ}$ is contained in the vector subspace
$\VV\times\UU\times\VV^{\perp}$ of
$\HH\oplus\HH\oplus\HH$.
Thus, applying \cite[Algorithm~12]{Combettes18} to the form
\cref{e:2bis} of \cref{e:2}
with:
\begin{itemize}
\item
the vector subspace $\mathsf{K}=\VV\times\UU\times\VV^{\perp}$,
\item
the activation blocks $(\forall n\in\NN)$ $I_n=\Xi_n$ and
$K_n=\Xi_n\cup\set{m+1}$,
\item
and the parameters
\begin{equation}
\begin{cases}
(\forall n\in\NN)\,\,\mu_{m+1,n}=1\\
(\forall\xi\in\Xi)(\forall n\in\NN)\,\,c_{\xi}(n)=n\\
(\forall k\in K)(\forall n\in\NN)\,\,d_k(n)=n,
\end{cases}
\end{equation}
\end{itemize}
and noticing that, in our setting:
\begin{itemize}
\item
the dual variables $(v_{k,n}^*)_{k\in K}$ in
\cite[Algorithm~12]{Combettes18} correspond to
$\brk{x_n^*(1),\ldots,x_n^*(m),v_n^*}$,
\item
the variables $(b_{k,n})_{k\in K}$ in
\cite[Algorithm~12]{Combettes18}
correspond to $(b_n(1),\ldots,b_n(m),p_n)$,
\end{itemize}
we obtain an algorithm of which the iteration $n$ update is
\begin{equation}
\label{e:nupd}
\begin{array}{l}
\text{for each scenario $\xi\in\Xi_n$}\\
\left\lfloor
\begin{array}{l}
l_n^*(\xi)
=\sum_{k\in\Xi}L_{k,\xi}^*\brk!{x_n^*(k)}+L_{m+1,\xi}^*v_n^*
\smallskip\\
a_n(\xi)=J_{\gamma_{\xi,n}A(\xi,\Cdot)}
\brk!{x_n(\xi)-\gamma_{\xi,n}l_n^*(\xi)}
\smallskip\\
a_n^*(\xi)=\gamma_{\xi,n}^{-1}\brk!{x_n(\xi)-a_n(\xi)}-l_n^*(\xi)
\end{array}
\right.\\
\text{for each scenario $\xi\in\Xi\smallsetminus\Xi_n$}\\
\left\lfloor
\begin{array}{l}
a_n(\xi)=a_{n-1}(\xi);\,\,
a_n^*(\xi)=a_{n-1}^*(\xi)
\end{array}
\right.
\smallskip\\
\text{for each $k\in K_n=\Xi_n\cup\set{m+1}$}\\
\left\lfloor
\begin{array}{l}
\text{if $k\in\Xi_n$}\\
\left\lfloor
\begin{array}{l}
l_n(k)=\sum_{\xi\in\Xi}L_{k,\xi}\brk!{x_n(\xi)}
\smallskip\\
b_n(k)=\proj_{C(k)}\brk!{l_n(k)+\mu_{k,n}x_n^*(k)}
\smallskip\\
b_n^*(k)=x_n^*(k)+\mu_{k,n}^{-1}\brk!{l_n(k)-b_n(k)}
\end{array}
\right.
\smallskip\\
\text{if $k=m+1$}\\
\left\lfloor
\begin{array}{l}
l_n(k)=\sum_{\xi\in\Xi}L_{k,\xi}\brk!{x_n(\xi)}
\smallskip\\
p_n=\proj_{\VV}\brk{l_n(k)+v_n^*}
\smallskip\\
p_n^*=v_n^*+l_n(k)-p_n
\end{array}
\right.\smallskip
\end{array}
\right.
\smallskip\\
\text{for each $k\in K\smallsetminus K_n=\Xi\smallsetminus\Xi_n$}\\
\left\lfloor
\begin{array}{l}
b_n(k)=b_{n-1}(k);\,\,
b_n^*(k)=b_{n-1}^*(k)
\end{array}
\right.
\smallskip\\
\begin{aligned}
&\textstyle
\brk2{\brk!{t_n^*(\xi)}_{\xi\in\Xi},
\brk!{u_n(\xi)}_{\xi\in\Xi},t_n}
=\proj_{\mathsf{K}}\bigg(
\brk2{a_n^*(\xi)+\sum_{k\in\Xi}L_{k,\xi}^*b_n^*(k)
+L_{m+1,\xi}^*p_n^*}_{\xi\in\Xi},
\\
&\textstyle\qquad\qquad\qquad\qquad
\brk!{b_n(k)-\sum_{\xi\in\Xi}L_{k,\xi}a_n(\xi)}_{k\in\Xi},
p_n-\sum_{\xi\in\Xi}L_{m+1,\xi}a_n(\xi)
\bigg)
\end{aligned}
\smallskip\\
\tau_n=\sum_{\xi\in\Xi}\brk!{\norm{t_n^*(\xi)}_{\HS_{\xi}}^2
+\norm{u_n(\xi)}_{\HS_{\xi}}^2}
+\norm{t_n}_{\HH}^2
\smallskip\\
\text{if $\tau_n>0$}\\
\left\lfloor
\begin{array}{l}
\begin{aligned}
&\textstyle
\kappa_n=
\sum_{\xi\in\Xi}\brk!{
\scal{x_n(\xi)}{t_n^*(\xi)}_{\HS_{\xi}}
-\scal{a_n(\xi)}{a_n^*(\xi)}_{\HS_{\xi}}}
\\
&\textstyle\quad\qquad
+\sum_{\xi\in\Xi}\brk!{\scal{u_n(\xi)}{x_n^*(\xi)}_{\HS_{\xi}}
-\scal{b_n(\xi)}{b_n^*(\xi)}_{\HS_{\xi}}}
+\scal{t_n}{v_n^*}_{\HH}
-\scal{p_n}{p_n^*}_{\HH}
\end{aligned}
\smallskip\\
\lambda_n\in\intv{\varepsilon}{2-\varepsilon}
\smallskip\\
\theta_n=\lambda_n\max\set{\kappa_n,0}/\tau_n
\end{array}
\right.\\
\text{else}\\
\left\lfloor
\begin{array}{l}
\theta_n=0
\end{array}
\right.\\
x_{n+1}=x_n-\theta_nt_n^*;\;
x_{n+1}^*=x_n^*-\theta_nu_n;\;
v_{n+1}^*=v_n^*-\theta_nt_n.
\end{array}
\end{equation}
Fix temporarily $n\in\NN$ and let us simplify \cref{e:nupd}.
It results from \cref{e:832y} that
\begin{equation}
\label{e:lnks}
(\forall k\in K_n)\quad
l_n(k)=
\begin{cases}
x_n(k),&\text{if}\,\,k\in\Xi_n;\\
x_n,&\text{if}\,\,k=m+1,
\end{cases}
\end{equation}
and that, by combining with \cite[Proposition~29.3]{Livre1}
and the fact that $\mathsf{K}=\VV\times\UU\times\VV^{\perp}$,
\begin{equation}
\label{e:utsz}
\brk[s]!{\,(\forall\xi\in\Xi)\,\,
u_n(\xi)=\proj_{U(\xi)}\brk!{b_n(\xi)-a_n(\xi)}\,}
\quad\text{and}\quad
t_n
=\proj_{\VV^{\perp}}(p_n-a_n)
={-}\proj_{\VV^{\perp}}a_n,
\end{equation}
where the last identity follows from the fact that $p_n\in\VV$.
Likewise, we derive from \cref{e:832z} that
\begin{equation}
\label{e:ekcz}
(\forall\xi\in\Xi_n)\quad
l_n^*(\xi)=x_n^*(\xi)+v_n^*(\xi)
\end{equation}
and that
\begin{equation}
t_n^*
=\proj_{\VV}(a_n^*+b_n^*+p_n^*)
=\proj_{\VV}(a_n^*+b_n^*),
\end{equation}
where the last identity follows from
\begin{equation}
\label{e:pns3}
p_n^*
=v_n^*+l_n(k)-p_n
=v_n^*+l_n(k)-\proj_{\VV}\brk{l_n(k)+v_n^*}
=\proj_{\VV^{\perp}}\brk{l_n(k)+v_n^*}
\in\VV^{\perp}.
\end{equation}
Thus, an induction argument shows that
\begin{equation}
\label{e:1037}
(x_n,x_n^*,v_n^*)\in\mathsf{K}=\VV\times\UU\times\VV^{\perp},
\end{equation}
and $x_n$ is therefore an implementable policy.
Hence, in view of \cref{e:lnks}, it follows from \cref{e:nupd} that
$p_n=x_n\in\VV$ and from \cref{e:pns3} that
$p_n^*=v_n^*\in\VV^{\perp}$.
At the same time, using \cref{e:utsz}, \cref{e:nupd},
and an induction argument, we obtain
\begin{equation}
(\forall n\in\NN)(\forall\xi\in\Xi)\quad
u_n(\xi)=
\begin{cases}
\proj_{U(\xi)}\brk!{b_n(\xi)-a_n(\xi)},&
\text{if}\,\,\xi\in\Xi_n;\\
u_{n-1}(\xi),&\text{otherwise}.
\end{cases}
\end{equation}
Further, since $\HH=\bigoplus_{\xi\in\Xi}\HS_{\xi}$
and since $\set{p_n}_{n\in\NN}\subset\VV$
and $\set{p_n^*}_{n\in\NN}\subset\VV^{\perp}$,
the computations of $(\tau_n)_{n\in\NN}$ and
$(\kappa_n)_{n\in\NN}$ in \cref{e:nupd}
can be rewritten as
\begin{equation}
(\forall n\in\NN)\quad
\begin{cases}
\tau_n=\norm{t_n^*}_{\HH}^2+\norm{u_n}_{\HH}^2+\norm{t_n}_{\HH}^2\\
\kappa_n=
\scal{x_n}{t_n^*}_{\HH}-\scal{a_n}{a_n^*}_{\HH}
+\scal{u_n}{x_n^*}_{\HH}-\scal{b_n}{b_n^*}_{\HH}
+\scal{t_n}{v_n^*}_{\HH}.
\end{cases}
\end{equation}
Altogether, upon rearranging the computation steps,
we see that \cref{e:nupd} simplifies to \cref{e:algo},
and we have thus shown that \cref{e:algo} is a realization of
\cite[Algorithm~12]{Combettes18} to \cref{e:2}.
Hence, \cite[Theorem~13]{Combettes18} asserts that there exists
$(\overline{x},\overline{x}^*,\overline{v}^*)\in\boldsymbol{\ZZ}$
such that $x_n\to\overline{x}$ and $v_n^*\to\overline{v}^*$,
while \cref{e:r46z} ensures that
$(\overline{x},\overline{v}^*)$ is a solution to \cref{e:1}.
\end{proof}

\begin{remark}
\label{r:2}
Let us comment on the core features of algorithm~\cref{e:algo}.
\begin{enumerate}
\item
\label{r:2i}
At each iteration $n$, only a user-chosen subset $\Xi_n$ of
scenarios needs to be activated.
The user can proactively decide which scenarios to activate based
on the available computational resources.
To guarantee convergence, we impose only
\cref{e:3kfs}, which roughly means that no scenario should be
``forgotten'' indefinitely long.
\item
\label{r:2ii}
For each activated scenario $\xi\in\Xi_n$, the computational effort
is split into separate evaluations of
the resolvent $J_{\gamma A(\xi,\Cdot)}$ and projectors
$\proj_{C(\xi)}$ and $\proj_{U(\xi)}$, which are often more
tractable. Such decomposition requires additional variables, the
cost of which can be justified by computational gains from avoiding
the need to iteratively solve complex constrained subproblems;
see \cref{sec:3} for further discussion.
\item
\label{r:2ii-bis}
The proposed method guarantees sure convergence
of the sequence $(x_n,v_n^*)_{n\in\NN}$ and not just
almost-sure convergence as in stochastic operator splitting methods;
see, e.g., \cite{Bareilles20,Combettes25,Combettes15} and the
references cited therein.
\item
\label{r:2iii}
Parallel to the discussion in \cite[Remark~5]{Combettes18},
the family $(U(\xi))_{\xi\in\Xi}$ may provide
computational and theoretical advantages.
For instance, if for some scenario $\xi\in\Xi$,
$C(\xi)$ is a vector subspace of $\RR^d$,
one can choose $U(\xi)=C(\xi)^{\perp}$,
which can reduce computational complexity by working in a
lower-dimensional space.
When a constraint set $C(\xi)$ has no special structure, one can
simply choose $U(\xi)=\RR^d$. In both examples, the condition
\cref{e:x2qw} is fulfilled.
See \cref{r:5} below for further discussion.
\item\label{r:2iv}
For the reader's convenience, we recall
the following formula from \cite{Rockafellar91} for evaluating the
projection $p=\proj_{\VV}x$ of a given $x\in\HH$ onto the
nonanticipativity subspace $\VV$:
\begin{equation}
\brk!{\forall k\in\set{1,\ldots,N}}
(\forall\xi\in\Xi)\quad
p_{[k]}(\xi)=\frac{\Sum_{\eta\sim_k\xi}\pi(\eta)x_{[k]}(\eta)}{
\Sum_{\eta\sim_k\xi}\pi(\eta)}.
\end{equation}
\end{enumerate}
\end{remark}

Continuing with the discussion in \cref{sec:1bis}
and \cref{r:2}\cref{r:2iii}, we next discuss in detail the
connection between our work and the projective hedging algorithm of
\cite{Eckstein25}.

\begin{remark}[connection with projective hedging algorithm]
\label{r:5}
In connection with \cref{r:2}\cref{r:2iii},
we now demonstrate a theoretical advantage of
the vector subspaces $(U(\xi))_{\xi\in\Xi}$ in \cref{t:1}
by showing that \cref{e:algo} encompasses a variant of
the projective hedging algorithm \cite[Algorithm~2]{Eckstein25}.
To this end, assume that $(\forall\xi\in\Xi)$ $C(\xi)=\RR^d$.
(This assumption covers the situation where
the resolvents of the sum $A(\xi,\Cdot)+N_{C(\xi)}$ can be reliably
evaluated and the constraint set $C(\xi)$ is therefore ``absorbed''
into the operator $A(\xi,\Cdot)$, as done, e.g., in
\cite{Bareilles20,Eckstein25}, in the optimization setting.)
We then choose
\begin{equation}
(\forall\xi\in\Xi)\quad
\begin{cases}
U(\xi)=\set{0}=\ran\brk!{\Id-\proj_{C(\xi)}}\\
x_0^*(\xi)=0\\
(\forall n\in\NN)\,\,\mu_{\xi,n}=1.
\end{cases}
\end{equation}
Hence $(\forall n\in\NN)$ $u_n=x_n^*=0$.
In turn, \cref{e:algo} reduces to
\begin{equation}
\label{e:nnpr}
\begin{array}{l}
\text{for $n=0,1,\ldots$}\\
\left\lfloor
\begin{array}{l}
\text{for each scenario $\xi\in\Xi_n$}\\
\left\lfloor
\begin{array}{l}
a_n(\xi)=J_{\gamma_{\xi,n}A(\xi,\Cdot)}\brk!{x_n(\xi)
-\gamma_{\xi,n}v_n^*(\xi)}
\smallskip\\
a_n^*(\xi)=\gamma_{\xi,n}^{-1}\brk!{x_n(\xi)-a_n(\xi)}
-v_n^*(\xi)
\end{array}
\right.\\
\text{for each scenario $\xi\in\Xi\smallsetminus\Xi_n$}\\
\left\lfloor
\begin{array}{l}
a_n(\xi)=a_{n-1}(\xi);\,\,
a_n^*(\xi)=a_{n-1}^*(\xi)
\end{array}
\right.\\
t_n^*=\proj_{\VV}a_n^*\\
t_n={-}\proj_{\VV^{\perp}}a_n\\
\tau_n=\norm{t_n^*}_{\HH}^2
+\norm{t_n}_{\HH}^2\\
\text{if $\tau_n>0$}\\
\left\lfloor
\begin{array}{l}
\kappa_n=\scal{x_n}{t_n^*}_{\HH}-\scal{a_n}{a_n^*}_{\HH}
+\scal{t_n}{v_n^*}_{\HH}\\
\lambda_n\in\intv{\varepsilon}{2-\varepsilon}\\
\theta_n=\lambda_n\max\set{\kappa_n,0}/\tau_n
\end{array}
\right.\\
\text{else}\\
\left\lfloor
\begin{array}{l}
\theta_n=0
\end{array}
\right.\\
x_{n+1}=x_n-\theta_nt_n^*;\;
v_{n+1}^*=v_n^*-\theta_nt_n.
\end{array}
\right.
\end{array}
\end{equation}
We observe that the computation of each $a_n(\xi)$ is exactly the same
as in the progressive hedging algorithm \cref{e:PH}.
Now assume that each $A(\xi,\Cdot)$ is the subdifferential of a lower
semicontinuous convex function $f(\xi,\Cdot)\colon\RR^d\to\RX$.
Then \cref{e:nnpr} strongly resembles the projective
hedging algorithm \cite[Algorithm~2]{Eckstein25},
with the following subtle differences:
\begin{itemize}
\item Evaluating $a_n(\xi)$ in \cref{e:nnpr} is equivalent to
computing the proximity operator
\begin{equation}
\prox_{\gamma_{\xi,n}f(\xi,\Cdot)}\brk!{x_n(\xi)-\gamma_{\xi,n}v_n^*(\xi)},
\end{equation}
which admits closed-form expressions for a wide array of proper
lower semicontinuous convex functions
\cite{Livre1,Combettes24}.
Here \cref{e:nnpr} requires full storage of decision vectors from
all $N$ stages, as in the classical progressive hedging algorithm.
\item The projective hedging algorithm does not require retaining the
last-stage variables between iterations.
However, it must solve generally non-strongly-convex subproblems
of the form
\begin{equation}
\minimize{\mathsf{x}\in\RR^d}{
f(\xi,\mathsf{x})
+\scal!{\mathsf{w}}{
\brk!{\mathsf{x}_{[1]},\ldots,\mathsf{x}_{[N-1]}}}_2
+\frac{1}{2\gamma}\norm!{
\brk!{\mathsf{x}_{[1]},\ldots,\mathsf{x}_{[N-1]}}
-\mathsf{z}}_2^2
},
\end{equation}
where $\mathsf{x}_{[k]}\in\RR^{d_k}$ denotes the $k$th-stage
decision block of $\mathsf{x}\in\RR^d$.
Such subproblems generally do not admit closed-form expressions for
their solutions, and therefore must be solved via iterative
methods. Additionally, they are clearly of different nature than
those of classical progressive hedging, contrary to the claim made
in the paragraph following \cite[Eqs.~(24) and (25)]{Eckstein25}.
\end{itemize}
\end{remark}

We end this section with potential further developments.

\begin{remark}[incorporating asynchrony]
For expository clarity, we have presented in \cref{t:1} only the
synchronous version of our algorithm.
It is evident from the proof that, since our approach is
based on the asynchronous block-activated method of
\cite{Combettes18}, asynchrony can be incorporated into
\cref{e:algo} in a straightforward manner.
\end{remark}

\begin{remark}[decomposition of cost operators]
We can extend \cref{prob:1} to handle more complex structures as
follows. Decompose each operator $A(\xi,\Cdot)$ as
\begin{multline}
A(\xi,\Cdot)=
T(\xi,\Cdot)+R(\xi,\Cdot)+Q(\xi,\Cdot),\\
\text{where}\,\,
\begin{cases}
\text{$T(\xi,\Cdot)\colon\RR^d\to 2^{\RR^d}$
is maximally monotone}\\
\text{$R(\xi,\Cdot)\colon\RR^d\to\RR^d$
is cocoercive}\\
\text{$Q(\xi,\Cdot)\colon\RR^d\to\RR^d$
is monotone and Lipschitzian}.
\end{cases}
\end{multline}
Then, upon combining the argument in the proof of \cref{t:1} and
the saddle projective splitting of \cite[Section~10]{Combettes24},
one can devise an asynchronous block-activated algorithm
for solving \cref{prob:1} which, at each iteration and for each
activated scenario $\xi\in\Xi$, evaluates
the resolvent $J_{\gamma T(\xi,\Cdot)}$,
the cocoercive operator
$R(\xi,\Cdot)$ (via an explicit Euler step), and the projector
$\proj_{C(\xi)}$ once,
while the Lipschitzian operator $Q(\xi,\Cdot)$ is evaluated twice
via an explicit Euler step.
\end{remark}

\section{Application: Solving multi-stage stochastic programming
  with \texorpdfstring{$\cvar$}{CVaR} objectives}
\label{sec:3}

In this section we demonstrate how the algorithm \cref{e:algo} can be
employed to solve the following multi-stage risk-averse stochastic
programming with conditional value-at-risk objectives.

\begin{problem}
  \label{prob:2}
  Consider \cref{m:1}. Let $\alpha\in\zeroun$ and denote by
  $\cvar_{\alpha}$
  the \emph{conditional value-at-risk at level $\alpha$}
  \cite{Rockafellar00}, that is,
  \begin{equation}
    \label{e:cvar}
    \cvar_{\alpha}\colon\GG\to\RR\colon
    z\mapsto\min_{\mathsf{y}\in\RR}\brk3{
      \mathsf{y}+\frac{1}{1-\alpha}\EE\brk!{
        \max\set{z-\mathsf{y},0}}
      },
  \end{equation}
  where $\GG=\set{z\colon\Xi\to\RR}$ and
  the expectation is taken with respect to the probability space
  $\Xi$.
  Additionally, for each scenario $\xi\in\Xi$,
  let $f(\xi,\cdot)\colon\RR^d\to\RR$ be convex
  and let $\emp\neq C(\xi)\subset\RR^d$ be closed and convex.
  The task is to
  \begin{equation}
    \label{e:9}
    \minimize{\substack{x\in\VV\\
        (\forall\xi\in\Xi)\,\,x(\xi)\in C(\xi)}}{
      \cvar_{\alpha}\brk!{f(\Cdot,x(\Cdot))}},
  \end{equation}
  under the assumption that a solution exists and that the following
  constraint qualification is satisfied:
  \begin{equation}
    \label{e:cqc1}
    (\exi x\in\VV)(\forall\xi\in\Xi)\quad
    x(\xi)\in\reli C(\xi),
  \end{equation}
  where $\reli C$ stands for the relative interior of a convex set
  $C\subset\RR^d$.
\end{problem}

We employ the techniques of \cite{Rockafellar18} to reformulate
\cref{prob:2} as an instantiation of \cref{prob:1},
thereby enabling the use of \cref{e:algo}.

\begin{proposition}
  \label{p:3}
  Consider the setting of \cref{prob:2}.
  Equip the set $\GG$ with the expectation scalar
  product
  \begin{equation}
    (y,z)\mapsto\EE(yz)=\sum_{\xi\in\Xi}\pi(\xi)y(\xi)z(\xi),
  \end{equation}
  let $\HHH=\GG\oplus\HH$ be the Hilbert direct sum of
  $\GG$ and $\HH$, and identify an element $(y,x)\in\HHH$ with the
  mapping
  \begin{equation}
    \Xi\to\RR\times\RR^d\colon
      \xi\mapsto\brk2{\brk!{y(\xi),x_{[1]}(\xi)},
        x_{[2]}(\xi),\ldots,x_{[N]}(\xi)}.
  \end{equation}
  We consider an augmented $N$-stage decision-making process in which:
  \begin{itemize}
    \item $\Xi$ (as in \cref{m:1}) is the probability space of all
      possible scenarios.
    \item $\HHH$ is the space of decision mappings.
    \item Given a decision mapping $(y,x)\in\HHH$,
      $\xi\mapsto(y(\xi),x_{[1]}(\xi))$ is the first-stage
      decision mapping and, for every $k\in\set{2,\ldots,N}$,
      $x_{[k]}$ is the $k$th-stage decision mapping.
  \end{itemize}
  Moreover, let $\VVV$ be the corresponding nonanticipativity
  subspace.
  Then the following hold:
  \begin{enumerate}
    \item\label{p:3i}
      Let $\WW$ be the vector subspace consisting of all constant
      mappings in $\GG$. Then $\VVV=\WW\times\VV$.
    \item\label{p:3ii}
      For each scenario $\xi\in\Xi$,
      let $\boldsymbol{A}(\xi,\Cdot)\colon\RR\times\RR^d\to
      2^{\RR\times\RR^d}$ be the subdifferential (with respect to
      the $\ell^2$ scalar product on $\RR\times\RR^d$) of
      \begin{equation}
        \boldsymbol{f}(\xi,\Cdot)\colon\RR\times\RR^d\to\RR\colon
        (\mathsf{y},\mathsf{x})\mapsto
        \mathsf{y}+\frac{1}{1-\alpha}\max\set{f(\xi,\mathsf{x})
          -\mathsf{y},0},
      \end{equation}
      and set
      $\boldsymbol{C}(\xi)=\RR\times C(\xi)$.
      Then the following hold:
      \begin{enumerate}
        \item\label{p:3iia}
          Let $\xi\in\Xi$. Then $\boldsymbol{A}(\xi,\Cdot)$ is maximally
          monotone.
        \item\label{p:3iib}
          Let $\overline{x}\in\HH$.
          Then $\overline{x}$ solves \cref{e:9} if and only if
          there exist $\overline{y}\in\GG$ and
          $\overline{\boldsymbol{v}}^*\in\VVV^{\perp}$
          such that
          $\overline{\boldsymbol{x}}=(\overline{y},\overline{x})$
          and $\overline{\boldsymbol{v}}^*$
          solve the stochastic equilibrium problem
          \begin{multline}
            \label{e:1+}
            \text{find $\boldsymbol{x}\in\VVV$ and
              $\boldsymbol{v}^*\in\VVV^{\perp}$ such that}\\
            (\forall\xi\in\Xi)\,\,
            {-}\boldsymbol{v}^*(\xi)\in
            \boldsymbol{A}\brk!{\xi,\boldsymbol{x}(\xi)}
            +N_{\boldsymbol{C}(\xi)}\boldsymbol{x}(\xi).
          \end{multline}
      \end{enumerate}
  \end{enumerate}
\end{proposition}
\begin{proof}
  \cref{p:3i}:
  Clear from the definition.

  \cref{p:3iia}:
  Note that $\boldsymbol{f}(\xi,\Cdot)$ is continuous as a
  real-valued convex function \cite[Corollary~8.40]{Livre1}.
  Thus $\boldsymbol{A}(\xi,\Cdot)$ is maximally monotone as the
  subdifferential of a continuous convex function
  \cite[Theorem~20.25]{Livre1}.

  \cref{p:3iib}:
  Set
  $\boldsymbol{\mathcal{C}}=\menge{(y,x)\in\HHH}{(\forall\xi\in\Xi)
    \,\,(y(\xi),x(\xi))\in\boldsymbol{C}(\xi)}$.
  Using calculus rules for relative interiors and \cref{e:cqc1}, we
  obtain $\VVV\cap\reli\boldsymbol{\mathcal{C}}\neq\emp$.
  Thus, since $\VVV$ is polyhedral as a vector subspace,
  it results from the sum rule \cite[Theorem~16.47(ii)]{Livre1} that
  \begin{equation}
    \label{e:sr33}
    N_{\VVV\cap\boldsymbol{\mathcal{C}}}
    =N_{\VVV}+N_{\boldsymbol{\mathcal{C}}}.
  \end{equation}
  At the same time, we derive from \cref{e:cvar} and \cref{p:3i} that
  \begin{align}
    &\min_{\substack{x\in\VV\\
        (\forall\xi\in\Xi)\,\,x(\xi)\in C(\xi)}}
    \cvar_{\alpha}\brk!{f(\Cdot,x(\Cdot))}
    \nonumber\\
    &\hspace{15mm}
    =\min_{\substack{\mathsf{y}\in\RR\\
        x\in\VV\\
        (\forall\xi\in\Xi)\,\,x(\xi)\in C(\xi)}}
    \EE\brk3{\mathsf{y}
      +\frac{1}{1-\alpha}
      \max\set!{f\brk!{\Cdot,x(\Cdot)}-\mathsf{y},0}}
    \nonumber\\
    &\hspace{15mm}
    =\min_{\substack{y\in\WW\\ x\in\VV\\
        (\forall\xi\in\Xi)\,\,x(\xi)\in C(\xi)}}
    \EE\brk3{y(\Cdot)+\frac{1}{1-\alpha}
      \max\set!{f\brk!{\Cdot,x(\Cdot)}-y(\Cdot),0}}
    \nonumber\\
    &\hspace{15mm}
    =\min_{(y,x)\in\VVV\cap\boldsymbol{\mathcal{C}}}
    \EE\brk2{\boldsymbol{f}\brk!{\Cdot,y(\Cdot),x(\Cdot)}}.
  \end{align}
  This argument also shows that
  $\overline{x}$ solves \cref{e:9} if and only if there exists
  $\overline{y}\in\GG$ such that $(\overline{y},\overline{x})$ solves
  \begin{equation}
    \label{e:61en}
    \minimize{(y,x)\in\VVV\cap\boldsymbol{\mathcal{C}}}{
    \EE\brk2{\boldsymbol{f}\brk!{\Cdot,y(\Cdot),x(\Cdot)}}}.
  \end{equation}
  Next, define
  $F\colon\HHH\to\RR\colon(y,x)\mapsto
  \EE\brk{\boldsymbol{f}\brk{\Cdot,y(\Cdot),x(\Cdot)}}$.
  Note that $F$ is convex and that, for every
  $(y,x)\in\HHH$ and every $(v,u)\in\HHH$, we have
  \begin{equation}
    (v,u)\in\partial F(y,x)
    \quad\Longleftrightarrow\quad
    (\forall\xi\in\Xi)\,\,
    \brk!{v(\xi),u(\xi)}\in\boldsymbol{A}\brk!{
      \xi,y(\xi),x(\xi)}
  \end{equation}
  and
  \begin{equation}
    (v,u)\in N_{\boldsymbol{\mathcal{C}}}(y,x)
    \quad\Longleftrightarrow\quad
    (\forall\xi\in\Xi)\,\,
    \brk!{v(\xi),u(\xi)}\in N_{\boldsymbol{C}(\xi)}\brk!{y(\xi),x(\xi)};
  \end{equation}
  see \cite[Proposition~16.9]{Livre1}.
  In turn, we derive from Fermat's rule
  (applied to \cref{e:61en}),
  sum rule \cite[Theorem~16.47(i)]{Livre1}, and \cref{e:sr33} that
  \begin{align}
    &\text{$\overline{x}$ solves \cref{e:9}}
    \nonumber\\
    &\hspace{2mm}
    \Longleftrightarrow
    (\exi\overline{y}\in\GG)\,\,
    \boldsymbol{0}
    \in\partial F(\overline{y},\overline{x})
    +N_{\VVV\cap\boldsymbol{\mathcal{C}}}(\overline{y},\overline{x})
    \nonumber\\
    &\hspace{2mm}
    \Longleftrightarrow
    (\exi\overline{y}\in\GG)\,\,
    \boldsymbol{0}
    \in\partial F(\overline{y},\overline{x})
    +N_{\VVV}(\overline{y},\overline{x})
    +N_{\boldsymbol{\mathcal{C}}}(\overline{y},\overline{x})
    \nonumber\\
    &\hspace{2mm}
    \Longleftrightarrow
    (\exi\overline{y}\in\GG)\,\,
    \begin{cases}
      (\overline{y},\overline{x})\in\VVV\\
      \brk!{\exi\overline{\boldsymbol{v}}^*\in\VVV^{\perp}}\,\,
      {-}\overline{\boldsymbol{v}}^*
    \in\partial F(\overline{y},\overline{x})
    +N_{\boldsymbol{\mathcal{C}}}(\overline{y},\overline{x})
    \end{cases}
    \nonumber\\
    &\hspace{2mm}
    \Longleftrightarrow
    (\exi\overline{y}\in\GG)\,\,
    \begin{cases}
      \overline{\boldsymbol{x}}=(\overline{y},\overline{x})\in\VVV\\
      \brk!{\exi\overline{\boldsymbol{v}}^*\in\VVV^{\perp}}
      (\forall\xi\in\Xi)\,\,
      {-}\overline{\boldsymbol{v}}^*(\xi)
      \in\boldsymbol{A}\brk!{\xi,\overline{\boldsymbol{x}}(\xi)}
      +N_{\boldsymbol{C}(\xi)}\brk!{\overline{\boldsymbol{x}}(\xi)},
    \end{cases}
  \end{align}
  which completes the proof.
\end{proof}

\begin{remark}
  \label{r:3}
  In the light of \cref{p:3}\cref{p:3iib},
  we can now apply the algorithm \cref{e:algo} to the stochastic
  equilibrium problem \cref{e:1+} to devise a provably-convergent
  block-activated algorithm for solving \cref{e:9},
  which evaluates $J_{\gamma\boldsymbol{A}(\xi,\Cdot)}$,
  $\proj_{C(\xi)}$, and $\proj_{\VV}$ separately.
  More precisely:
  \begin{itemize}
    \item
      The application of \cref{e:algo} to \cref{e:1+} requires
      evaluating $\proj_{\boldsymbol{C}(\xi)}$ and $\proj_{\VVV}$.
      However, the direct sum structure of $\HHH$ yields a natural
      decomposition in implementing these projectors.
      Indeed, thanks to \cite[Proposition~29.3]{Livre1},
      \begin{equation}
        (\forall\xi\in\Xi)\brk!{\forall(\mathsf{y},\mathsf{x})\in\RR\times\RR^d}\quad
        \proj_{\boldsymbol{C}(\xi)}(\mathsf{y},\mathsf{x})
        =(\mathsf{y},\proj_{C(\xi)}\mathsf{x})
      \end{equation}
      and
      \begin{equation}
        \brk!{\forall(y,x)\in\HHH}\quad
        \proj_{\VVV}(y,x)=(\proj_{\WW}y,\proj_{\VV}x).
      \end{equation}
      The closed-form expression for $\proj_{\VV}$ is stated in
      \cref{r:2}\cref{r:2iv} and projecting onto $\WW$ is just an
      averaging operation:
      \begin{equation}
        (\forall y\in\GG)\quad
        \proj_{\WW}y\colon\Xi\to\RR\colon
        \xi\mapsto\sum_{\eta\in\Xi}\pi(\eta)y(\eta).
      \end{equation}
    \item
      We have
      $J_{\gamma\boldsymbol{A}(\xi,\Cdot)}
      =\prox_{\gamma\boldsymbol{f}(\xi,\Cdot)}$, which can be evaluated
      through solving a univariate equation, as described in
      \cref{p:4} below.
  \end{itemize}
  By contrast, Rockafellar's
  \cite[Progressive hedging algorithm 2]{Rockafellar18}
  must, at each iteration, activate all scenarios and solve
  constrained minimization subproblems of the form
  \begin{equation}
    \minimize{\substack{\mathsf{y}\in\RR\\
        \mathsf{x}\in C(\xi)}}{
      \mathsf{y}+\frac{1}{1-\alpha}\max\set!{
        f\brk{\xi,\mathsf{x}}-\mathsf{y},0}
      +\mathsf{u}\mathsf{y}
      +\scal{\mathsf{w}}{\mathsf{x}}
      +\frac{1}{2\gamma}\brk!{\abs{\mathsf{y}-\mathsf{b}}^2
        +\norm{\mathsf{x}-\mathsf{z}}^2}
      },
  \end{equation}
  where $\set{\mathsf{u},\mathsf{b}}\subset\RR$ and
  $\set{\mathsf{w},\mathsf{z}}\subset\RR^d$ are given.
\end{remark}

Establishing \cref{p:4} necessitates the following lemma, which is based
on ideas found in \cite{Combettes24b}.

\begin{lemma}
  \label{l:1}
  Let $\HH$ be a Euclidean space,
  let $f\colon\HH\to\RR$ be convex,
  let $\gamma\in\RPP$, let $x\in\HH$,
  and set $p=\prox_{\gamma\max\set{f,0}}x$.
  Then exactly one of the following holds:
  \begin{enumerate}
    \item\label{l:1i}
      $f(x)\leq 0$, in which case $p=x$.
    \item\label{l:1ii}
      $f(\prox_{\gamma f}x)>0$, in which case
      $p=\prox_{\gamma f}x$.
    \item\label{l:1iii}
      $f(x)>0$ and $f(\prox_{\gamma f}x)\leq 0$, in which case,
      the system of equations
      \begin{equation}
        \begin{cases}
          \label{e:wbn8}
          \theta\in\intv[l]{0}{1}\\
          f\brk!{\prox_{\theta\gamma f}x}=0
        \end{cases}
      \end{equation}
      has a solution and, for every $\theta$ satisfying \cref{e:wbn8},
      we have $p=\prox_{\theta\gamma f}x$.
  \end{enumerate}
\end{lemma}
\begin{proof}
  Thanks to the inequality
  $\norm{x-\prox_{\gamma f}x}^2+\gamma f\brk{\prox_{\gamma f}x}
  \leq\gamma f(x)$ \cite[Proposition~12.27]{Livre1},
  the cases \cref{l:1i} and \cref{l:1ii} are mutually exclusive.
  Now set $\phi=\max\set{\Cdot,0}$.
  Since $\phi=d_{\intv[l]{\minf}{0}}$,
  we deduce from \cite[Example~16.62]{Livre1} that
  \begin{equation}
    \label{e:fp42}
    (\forall\kappa\in\RR)\quad
    \partial\phi(\kappa)=
    \begin{cases}
      \set{1},&\text{if}\,\,\kappa>0;\\
      \intv{0}{1},&\text{if}\,\,\kappa=0;\\
      \set{0},&\text{if}\,\,\kappa<0.
    \end{cases}
  \end{equation}
  In turn, we derive from
  \cite[Proposition~16.44 and Corollary~16.72]{Livre1}
  that
  \begin{align}
    \label{e:qf37}
    (\forall q\in\HH)\quad
    q=\prox_{\gamma\phi\circ f}x
    &\Longleftrightarrow
    \gamma^{-1}(x-q)\in\partial(\phi\circ f)(q)
    \nonumber\\
    &\Longleftrightarrow
    \brk!{\exi\theta\in\partial\phi\brk!{f(q)}}\,\,
    \gamma^{-1}(x-q)\in\theta\partial f(q)
    \nonumber\\
    &\Longleftrightarrow
    \brk!{\exi\theta\in\partial\phi\brk!{f(q)}}\,\,
    q=\prox_{\theta\gamma f}x.
  \end{align}

  \cref{l:1i}:
  In this case \cref{e:qf37} is satisfied with $q=x$ and $\theta=0$.

  \cref{l:1ii}:
  In this case \cref{e:qf37} is satisfied with
  $q=\prox_{\gamma f}x$ and $\theta=1$.

  \cref{l:1iii}:
  Applying \cite[Proposition~23.31(iii)]{Livre1} to the maximally
  monotone operator $\partial f$, we deduce that
  $\RPP\to\HH\colon\kappa\mapsto\prox_{\kappa f}x$ is continuous.
  At the same time, since $\HH$ is finite-dimensional
  and $f$ is real-valued and convex,
  $f$ is continuous \cite[Corollary~8.40]{Livre1}.
  Thus, the function
  $\psi\colon\RPP\to\RR\colon\kappa\mapsto f(\prox_{\kappa f}x)$
  is continuous and decreasing \cite[Proposition~12.33(iii)]{Livre1}
  with $\psi(\kappa)\uparrow f(x)>0$ as $\kappa\downarrow 0$.
  Hence, the intermediate value theorem guarantees that
  \cref{e:wbn8} has a solution.
  The remaining claim is a consequence of \cref{e:qf37},
  \cref{e:fp42}, and the uniqueness of proximal points.
\end{proof}

\begin{proposition}
  \label{p:4}
  Let $f\colon\RR^d\to\RR$ be convex,
  let $\alpha\in\zeroun$, and set
  \begin{equation}
    \boldsymbol{f}\colon\RR\times\RR^d\to\RR\colon
    (\mathsf{y},\mathsf{x})\mapsto
    \mathsf{y}+\frac{1}{1-\alpha}\max\set{f(\mathsf{x})
      -\mathsf{y},0}.
  \end{equation}
  Additionally,
  let $\gamma\in\RPP$,
  set $\tau=\gamma/(1-\alpha)$,
  let $(\mathsf{y},\mathsf{x})\in\RR\times\RR^d$,
  and set
  $(\mathsf{q},\mathsf{p})
  =\prox_{\gamma\boldsymbol{f}}(\mathsf{y},\mathsf{x})$.
  Then exactly one of the following holds:
  \begin{enumerate}
    \item\label{p:4i}
      $f(\mathsf{x})-\mathsf{y}+\gamma\leq 0$,
      in which case
      $(\mathsf{q},\mathsf{p})=(\mathsf{y}-\gamma,\mathsf{x})$.
    \item\label{p:4ii}
      $f(\prox_{\tau f}\mathsf{x})-\mathsf{y}>\tau-\gamma$,
      in which case
      $(\mathsf{q},\mathsf{p})=(\mathsf{y}-\gamma+\tau,
      \prox_{\tau f}\mathsf{x})$.
    \item\label{p:4iii}
      $f(\mathsf{x})-\mathsf{y}+\gamma>0$ and
      $f(\prox_{\tau f}\mathsf{x})-\mathsf{y}\leq\tau-\gamma$,
      in which case
      the system of equations
      \begin{equation}
        \begin{cases}
          \label{e:wbn9}
          \theta\in\intv[l]{0}{1}\\
          f\brk!{\prox_{\theta\tau f}\mathsf{x}}-\mathsf{y}+\gamma
          -\theta\tau=0
        \end{cases}
      \end{equation}
      has a solution and, for every $\theta$ satisfying \cref{e:wbn9},
      $(\mathsf{q},\mathsf{p})=(\mathsf{y}-\gamma+\theta\tau,
      \prox_{\theta\tau f}\mathsf{x})$.
  \end{enumerate}
\end{proposition}
\begin{proof}
  Set $\phi=\max\set{\Cdot,0}$ and
  $\boldsymbol{g}\colon\RR\times\RR^d\to\RR\colon
  (\mathsf{y},\mathsf{x})\mapsto f(\mathsf{x})-\mathsf{y}$.
  Then, \cite[Proposition~24.11]{Livre1} gives
  \begin{equation}
    \brk!{\forall\kappa\in\RPP}
    \brk!{\forall(\mathsf{b},\mathsf{a})\in\RR\times\RR^d}\quad
    \prox_{\kappa\boldsymbol{g}}(\mathsf{b}-\gamma,\mathsf{a})
    =(\mathsf{b}-\gamma+\kappa,\prox_{\kappa f}\mathsf{a}).
  \end{equation}
  At the same time, because
  $\gamma\boldsymbol{f}
  =\scal{(\mathsf{y},\mathsf{x})}{(\gamma,\mathsf{0})}+
  \tau\phi\circ\boldsymbol{g}$,
  it results from
  \cite[Proposition~24.8(i)]{Livre1} that
  $(\mathsf{q},\mathsf{p})
    =\prox_{\tau\phi\circ\boldsymbol{g}}(\mathsf{y}-\gamma,\mathsf{x})$.
  Now apply \cref{l:1}.
\end{proof}

We conclude this section by outlining an additional application of the
proposed algorithmic framework.

\begin{remark}
\label{r:8}
By combining the reformulation technique of Rockafellar and Uryasev in
\cite{Rockafellar20} and the proposed algorithmic framework, we can
devise an asynchronous block-activated decomposition method with
the features \cref{G1,G2,G3} for solving constrained
multi-stage stochastic programming with \emph{buffered probability
of exceedance} objective functions.
\end{remark}

\section{Numerical experiments}
\label{sec:5}

As discussed in \cref{cf:1,cf:2}, the computational efficiency of
the progressive hedging algorithm \cref{e:PH} depends critically on
the efficient evaluation, at every iteration, of all the resolvents
\begin{equation}
\brk!{J_{\gamma\brk{A(\xi,\Cdot)+N_{C(\xi)}}}}_{\xi\in\Xi}.
\end{equation}
Consequently, a substantial portion of the computational cost
of the progressive hedging algorithm is devoted to solving
these scenario subproblems.

The aim of this section is therefore twofold:
\begin{itemize}
\item
To demonstrate the computational gains enabled by the features
\cref{G1,G2} of the proposed algorithm \cref{e:algo}.
\item
To compare the numerical performance of \cref{e:algo},
under both full and partial scenario activation,
with that of the progressive hedging algorithm \cref{e:PH}.
\end{itemize}
For these purposes, we revisit the class of two-stage stochastic
affine variational inequalities considered in the numerical study
of Rockafellar and Sun \cite{Rockafellar19}. These test problems
are instantiations of \cref{prob:1} with the following
specifications:
\begin{itemize}
\item
The number of stages is $N=2$ and the scenario set $\Xi$ is of the
form $\Xi=\Xi_1\times\Xi_2$.
The probabilities $(\pi(\xi))_{\xi\in\Xi}$ are obtained by first
sampling a vector uniformly from $\intv{0.5}{1.5}^{\Xi}$
and then normalizing its entries so that
$\sum_{\xi\in\Xi}\pi(\xi)=1$.
\item
For each scenario $\xi\in\Xi$,
\begin{equation}
A(\xi,\Cdot)=M(\xi)\Cdot+q(\xi)
\quad\text{and}\quad
C(\xi)=\RP^d,
\end{equation}
where $\RR^d\ni q(\xi)\sim\mathcal{N}(0,\Id)$
and $M(\xi)\in\RR^{d\times d}$ is a matrix generated as follows:
\begin{itemize}
\item
Randomly generate a symmetric positive semidefinite matrix
$M_{\mathsf{S}}(\xi)\in\RR^{d\times d}$ of rank $\lceil 3d/4\rceil$.
\item
Randomly generate a nonzero skew-symmetric matrix
$M_{\mathsf{A}}(\xi)\in\RR^{d\times d}$.
\item
Set $M(\xi)=M_{\mathsf{S}}(\xi)+\alpha M_{\mathsf{A}}(\xi)$,
where $\alpha=0.2$.
\end{itemize}
Thus $M(\xi)$ defines a monotone---generally not
self-adjoint---linear operator.
\end{itemize}
We consider a fixed-size benchmark
\begin{equation}
d_1=d_2=150,\quad
\card{\Xi_1}=50,\quad
\text{and}\quad
\card{\Xi_2}=40,
\end{equation}
and apply both algorithms to the same 10 randomly generated
instances. With this choice,
the experiment requires nontrivial computational effort
while remaining practical to execute on a personal laptop.
Furthermore, the scenario set $\Xi$ contains
$(\card{\Xi_1})\times(\card{\Xi_2})=2000$ scenarios.
Consequently, the progressive hedging algorithm \cref{e:PH} must,
at every iteration, solve 2,000 scenario subproblems, each
corresponding to a linear complementarity problem
involving a dense, unstructured $300\times 300$ matrix.

\paragraph{\textcolor{structure}{Implementation of the proposed
algorithm \cref{e:algo}}}{%
The algorithm is initialized at the zero vectors.
We choose
\begin{equation}
(\forall\xi\in\Xi)\quad
U(\xi)=\RR^d,
\end{equation}
and the parameters $(\gamma_{\xi,n})_{\xi\in\Xi,n\in\NN}$,
$(\mu_{\xi,n})_{\xi\in\Xi,n\in\NN}$, and
$(\lambda_n)_{n\in\NN}$ are kept constant across
scenarios and iterations, namely
\begin{equation}
(\forall\xi\in\Xi)(\forall n\in\NN)\quad
\gamma_{\xi,n}=\gamma,
\quad
\mu_{\xi,n}=\mu,
\quad\text{and}\quad
\lambda_n=\lambda.
\end{equation}
The values of $\gamma$, $\mu$, and $\lambda$ were selected through
preliminary experimentation to provide good overall performance,
and then kept fixed throughout all reported experiments. In
particular:
\begin{equation}
\gamma=\frac{1}{\sqrt{d}},\quad
\mu=1,
\quad\text{and}\quad
\lambda=1.
\end{equation}
(The value $\gamma=1/\sqrt{d}$ was inspired by the heuristic
proposed in \cite{Rockafellar19} for the implementation of the
progressive hedging algorithm.)
To evaluate each resolvent $J_{\gamma A(\xi,\Cdot)}$ in the proposed
algorithm \cref{e:algo}, note that
\begin{equation}
(\forall\mathsf{x}\in\RR^d)\quad
J_{\gamma A(\xi,\Cdot)}\mathsf{x}
=\brk!{\Id+\gamma M(\xi)}^{-1}\brk!{\mathsf{x}-\gamma q(\xi)}.
\end{equation}
Thus, for each scenario $\xi\in\Xi$, we compute an
$\mathsf{LU}$-factorization of $\Id+\gamma M(\xi)$
once during the initialization phase.
The resulting $\mathsf{LU}$-factorizations are then reused at every
subsequent iteration. Hence, each evaluation of $J_{\gamma
A(\xi,\Cdot)}$ reduces to solving only two \emph{triangular} linear
systems. Finally, for every $\delta\in\set{0.2,0.4,0.6,0.8,1.0}$,
the proposed algorithm is implemented with a scenario-activation
strategy satisfying $\Xi_0=\Xi$ and
$(\forall n\in\NN\smallsetminus\set{0})$
$\card\Xi_n=\lceil\delta\card\Xi\rceil$.
More precisely, identify
\begin{equation}
\Xi=\set{1,\ldots,\card\Xi},
\end{equation}
let $\delta\in\set{0.2,0.4,0.6,0.8}$
(the case $\delta=1.0$ is straightforward),
and set
\begin{equation}
m_{\delta}=\lceil 0.2\delta\card\Xi\rceil.
\end{equation}
For every $n\in\NN\smallsetminus\set{0}$,
the scenario block $\Xi_n$ is chosen as follows:
\begin{description}
\item[Step 1]
Set
\begin{equation}
\Xi_n^{\mathrm{cyc}}=
\menge{1+\rem\brk!{j+(n-1)m_{\delta}}}{
j\in\set{0,\ldots,m_{\delta}-1}},
\end{equation}
where $\rem a$ denotes the remainder of the division of $a\in\NN$
by $\card\Xi$.
\item[Step 2]
For every $\xi\in\Xi$, set
\begin{multline}
\chi_n(\xi)=\pi(\xi)\bigl(
\scal{x_n(\xi)-a_{n-1}(\xi)}{a_{n-1}^*(\xi)+x_n^*(\xi)
+v_n^*(\xi)}\\
+\scal{x_n(\xi)-b_{n-1}(\xi)}{b_{n-1}^*(\xi)-x_n^*(\xi)}
\bigr).
\end{multline}
Construct a set
$\Xi_n^{\chi}$ of largest possible cardinality by selecting,
among the scenarios in \begin{equation}
\menge{\xi\in\Xi\smallsetminus\Xi_n^{\mathrm{cyc}}}{\chi_n(\xi)<0},
\end{equation}
up to $\lceil\delta\card\Xi\rceil-m_{\delta}$ scenarios
$\xi$ with the most negative values of $\chi_n(\xi)$;
if two scenarios $\xi$ and $\eta$ satisfy
$\chi_n(\xi)=\chi_n(\eta)$, we choose the one with the smaller
index in $\Xi$.
This selection strategy was inspired by the implementation of
the projective hedging algorithm in \cite{Eckstein25}
(see also \cite[Section~5.3]{Johnstone22}).
\item[Step 3]
If fewer than
$\lceil\delta\card\Xi\rceil$ scenarios have been selected
after Steps~1 and 2, construct $\Xi_n^{\mathrm{rem}}$
by selecting exactly
\begin{equation}
\lceil\delta\card\Xi\rceil
-\card\Xi_n^{\mathrm{cyc}}
-\card\Xi_n^{\chi}
\end{equation}
scenarios from
$\Xi\smallsetminus(\Xi_n^{\mathrm{cyc}}\cup\Xi_n^{\chi})$,
prioritizing those that have not been activated for the
longest time. Otherwise, set $\Xi_n^{\mathrm{rem}}=\emp$.
\item[Step 4]
Set $\Xi_n=\Xi_n^{\mathrm{cyc}}\cup\Xi_n^{\chi}\cup
\Xi_n^{\mathrm{rem}}$.
\end{description}
This scenario-activation strategy ensures that,
for every $\delta\in\set{0.2,0.4,0.6,0.8}$,
\cref{e:3kfs} is satisfied with
$T=\lceil(\card\Xi)/m_{\delta}\rceil-1$.
}

\paragraph{\textcolor{structure}{Implementation of the progressive
hedging algorithm \cref{e:PH}}}{%
The algorithm is initialized at the zero vectors, and the parameter
$\gamma$ was selected through preliminary
experimentation---including the heuristic $\gamma=1/\sqrt{d}$
considered in \cite{Rockafellar19}---to
provide good overall performance, and then kept fixed throughout
all reported experiments; specifically,
\begin{equation}
\gamma=\frac{1}{\sqrt{d}}.
\end{equation}
At each iteration and for each scenario
$\xi\in\Xi$, evaluating the resolvent
$J_{\gamma\brk{A(\xi,\Cdot)+N_{C(\xi)}}}$
at a point $\mathsf{x}\in\RR^d$ amounts to finding a zero of
\begin{equation}
T_{\mathsf{x}}(\xi,\Cdot)\colon\RR^d\to\RR^d\colon
\mathsf{z}\mapsto\mathsf{z}-\proj_{\RP^d}\brk!{
\mathsf{z}-M(\xi)\mathsf{z}-\gamma^{-1}\mathsf{z}
+\gamma^{-1}\mathsf{x}
-q(\xi)}.
\end{equation}
Following \cite{Rockafellar19}, we solve the subproblem
\begin{equation}
\text{find}\,\,\mathsf{z}\in\RR^d\,\,\text{such that}\,\,
T_{\mathsf{x}}(\xi,\mathsf{z})=0
\end{equation}
by using the semismooth Newton method of \cite{Sun93},
together with the warm-start strategy described in
\cite{Rockafellar19}.
}

\paragraph{\textcolor{structure}{Stopping criterion}}{%
Since the operators $(A(\xi,\Cdot))_{\xi\in\Xi}$ are single-valued
in the present setting, we deduce that,
for every $x\in\VV$ and every $v^*\in\VV^{\perp}$,
\begin{align}
&\text{$(x,v^*)$ solves \cref{e:1}}
\nonumber\\
&\hskip 5mm
\Longleftrightarrow
(\forall\xi\in\Xi)\,\,
{-}v^*(\xi)-A\brk!{\xi,x(\xi)}\in N_{C(\xi)}x(\xi)
\nonumber\\
&\hskip 5mm
\Longleftrightarrow
\sum_{\xi\in\Xi}\pi(\xi)\norm!{x(\xi)
-\proj_{C(\xi)}\brk!{x(\xi)-v^*(\xi)-A\brk!{\xi,x(\xi)}}
}_2^2=0.
\end{align}
Moreover, because the sequence
$(x_n,v_n^*)_{n\in\NN}$ in both algorithms
\cref{e:algo,e:PH} satisfy
\begin{equation}
\begin{cases}
\set{x_n}_{n\in\NN}\subset\VV,\,\,
\set{v_n^*}_{n\in\NN}\subset\VV^{\perp}\\
\text{$(x_n,v_n^*)_{n\in\NN}$ converges to a solution to
\cref{e:1}},
\end{cases}
\end{equation}
it follows from the continuity of $(A(\xi,\Cdot))_{\xi\in\Xi}$ that
\begin{equation}
\sum_{\xi\in\Xi}\pi(\xi)\norm!{x_n(\xi)
-\proj_{C(\xi)}\brk!{x_n(\xi)-v_n^*(\xi)-A\brk!{\xi,x_n(\xi)}}
}_2^2\to 0.
\end{equation}
Thus, we terminate \cref{e:algo,e:PH} when
\begin{equation}
\label{e:crit}
\sqrt{\sum_{\xi\in\Xi}\pi(\xi)\norm!{x_n(\xi)
-\proj_{C(\xi)}\brk!{x_n(\xi)-v_n^*(\xi)-A\brk!{\xi,x_n(\xi)}}
}_2^2}
\leq 10^{-5},
\end{equation}
or when a prescribed maximum number of iterations has been reached.
A run is declared \emph{successful} if the algorithm satisfies
\cref{e:crit} within the prescribed maximum of 8,000 outer iterations.
}

\paragraph{\textcolor{structure}{Results}}{%
All experiments were implemented in Python~3.9.6
using NumPy~2.0.2 and SciPy~1.13.1, and executed on a MacBook Pro
M4 with 24~GB of RAM.
Both algorithms were implemented using standard serial loops
over the scenarios, without algorithm-specific parallelization.
For each of the algorithms
\cref{e:algo} and \cref{e:PH},
we record two measures of computation time:
\begin{itemize}
\item
The \emph{wall-clock time},
defined as the total elapsed time from algorithm initialization
to termination, \emph{including} the evaluation
of
\begin{equation}
\label{e:eval}
\sqrt{\sum_{\xi\in\Xi}\pi(\xi)\norm!{x_n(\xi)
-\proj_{C(\xi)}\brk!{x_n(\xi)-v_n^*(\xi)-A\brk!{\xi,x_n(\xi)}}
}_2^2}
\end{equation}
after every outer iteration, as required by the stopping criterion
\cref{e:crit}.
\item
The \emph{algorithmic computation time},
obtained by subtracting from the wall-clock time the cumulative
time spent on evaluating \cref{e:eval}.
For the proposed algorithm \cref{e:algo},
this measure includes the one-time
$\mathsf{LU}$-factorizations performed during initialization.
\end{itemize}
All methods successfully solved all ten instances.
For each of these two measures, we report
the minimum, first and third quartiles (that is, Q1 and Q3,
respectively), median, average, and maximum over the ten runs;
see \cref{tab:wallclock,tab:algtime}.
Because both algorithms were executed on identical randomly
generated instances, we additionally report in \cref{tab:ratio} the
paired runtime ratios
\begin{equation}
\frac{T_{\cref{e:PH}}}{T_{\cref{e:algo}_{\delta}}},
\end{equation}
where $T_{\cref{e:PH}}$ denotes the runtime of \cref{e:PH}
and $T_{\cref{e:algo}_{\delta}}$ denotes the
runtime of \cref{e:algo} with activation fraction
$\delta\in\set{0.2,0.4,0.6,0.8,1.0}$.
These ratios are computed separately for wall-clock time
and algorithmic computation time.
Finally, to facilitate the discussion, we also report in
\cref{tab:iter} the statistics of the number of outer iterations
over the ten runs.

\begin{table}[!htbp]
\centering
\begin{tabular}{lccccccc}
\toprule
\multicolumn{2}{c}{\multirow{2}{*}{Method}} &
\multicolumn{6}{c}{Wall-clock time (s)}\\
\cmidrule(lr){3-8}
&
&
Min &
Q1 &
Median &
Average &
Q3 &
Max\\
\midrule
\multicolumn{2}{l}{Progressive hedging \cref{e:PH}} &
415.374 &
419.668 &
422.729 &
422.827 &
425.339 &
433.891 \\
\midrule
\multirow{5}{*}{Alg.~\cref{e:algo} with} &
$\delta=1.0$ &
106.115 &
127.503 &
140.657 &
141.721 &
159.711 &
172.941 \\
\cmidrule(lr){2-8}
& $\delta=0.8$ &
112.896 &
142.643 &
145.435 &
151.893 &
163.372 &
208.088 \\
\cmidrule(lr){2-8}
& $\delta=0.6$ &
122.578 &
154.278 &
161.210 &
168.279 &
187.437 &
219.802 \\
\cmidrule(lr){2-8}
& $\delta=0.4$ &
147.384 &
183.088 &
208.680 &
210.764 &
243.434 &
261.955 \\
\cmidrule(lr){2-8}
& $\delta=0.2$ &
244.100 &
288.805 &
367.832 &
341.966 &
387.419 &
408.021 \\
\bottomrule
\end{tabular}
\caption{Wall-clock time statistics over the ten runs.}
\label{tab:wallclock}
\end{table}

\begin{table}[ht]
\centering
\begin{tabular}{lccccccc}
\toprule
\multicolumn{2}{c}{\multirow{2}{*}{Method}} &
\multicolumn{6}{c}{Algorithmic computation time (s)}\\
\cmidrule(lr){3-8}
&
&
Min &
Q1 &
Median &
Average &
Q3 &
Max\\
\midrule
\multicolumn{2}{l}{Progressive hedging \cref{e:PH}} &
403.187 &
407.608 &
410.774 &
410.848 &
413.280 &
422.348 \\
\midrule
\multirow{5}{*}{Alg.~\cref{e:algo} with} &
$\delta=1.0$ &
75.124 &
90.234 &
99.495 &
100.242 &
112.848 &
122.336 \\
\cmidrule(lr){2-8}
& $\delta=0.8$ &
76.083 &
96.109 &
97.945 &
102.279 &
109.998 &
139.764 \\
\cmidrule(lr){2-8}
& $\delta=0.6$ &
75.532 &
94.851 &
99.124 &
103.416 &
115.057 &
134.933 \\
\cmidrule(lr){2-8}
& $\delta=0.4$ &
77.911 &
96.614 &
109.992 &
111.107 &
128.202 &
137.909 \\
\cmidrule(lr){2-8}
& $\delta=0.2$ &
94.332 &
111.297 &
141.590 &
131.636 &
149.022 &
156.846 \\
\bottomrule
\end{tabular}
\caption{Algorithmic computation time statistics over the ten runs.}
\label{tab:algtime}
\end{table}

\begin{table}[!htbp]
\centering
\begin{tabular}{lccccccc}
\toprule
\multirow{2}{*}{Measure} &
\multirow{2}{*}{$\delta$} &
\multicolumn{6}{c}{Paired runtime ratio
$T_{\cref{e:PH}}/T_{\cref{e:algo}_{\delta}}$}
\\
\cmidrule(lr){3-8}
&
&
Min &
Q1 &
Median &
Average &
Q3 &
Max\\
\midrule
\multirow{5}{*}{Wall-clock time} &
1.0 &
2.498 &
2.623 &
3.010 &
3.050 &
3.334 &
3.981 \\
\cmidrule(lr){2-8}
&
0.8 &
2.085 &
2.564 &
2.900 &
2.851 &
2.987 &
3.742 \\
\cmidrule(lr){2-8}
&
0.6 &
1.974 &
2.254 &
2.618 &
2.578 &
2.763 &
3.446 \\
\cmidrule(lr){2-8}
&
0.4 &
1.586 &
1.728 &
2.036 &
2.073 &
2.324 &
2.866 \\
\cmidrule(lr){2-8}
&
0.2 &
1.037 &
1.079 &
1.160 &
1.284 &
1.501 &
1.727 \\
\midrule
\multirow{5}{*}{Algorithmic time} &
1.0 &
3.432 &
3.604 &
4.134 &
4.189 &
4.577 &
5.464 \\
\cmidrule(lr){2-8}
&
0.8 &
3.022 &
3.697 &
4.183 &
4.113 &
4.309 &
5.395 \\
\cmidrule(lr){2-8}
&
0.6 &
3.130 &
3.566 &
4.136 &
4.075 &
4.368 &
5.434 \\
\cmidrule(lr){2-8}
&
0.4 &
2.924 &
3.187 &
3.752 &
3.819 &
4.280 &
5.268 \\
\cmidrule(lr){2-8}
&
0.2 &
2.621 &
2.722 &
2.932 &
3.239 &
3.783 &
4.343 \\
\bottomrule
\end{tabular}
\caption{Statistics of the paired runtime ratio
$T_{\cref{e:PH}}/T_{\cref{e:algo}_{\delta}}$.}
\label{tab:ratio}
\end{table}

\begin{table}[ht]
\centering
\begin{tabular}{lccccccc}
\toprule
\multicolumn{2}{c}{\multirow{2}{*}{Method}} &
\multicolumn{6}{c}{Number of outer iterations}\\
\cmidrule(lr){3-8}
&
&
Min &
Q1 &
Median &
Average &
Q3 &
Max\\
\midrule
\multicolumn{2}{l}{Progressive hedging \cref{e:PH}} &
304 &
305 &
306 &
305.700 &
306 &
307 \\
\midrule
\multirow{5}{*}{Alg.~\cref{e:algo} with} &
$\delta=1.0$ &
779 &
945.250 &
1041.500 &
1048.700 &
1173 &
1296 \\
\cmidrule(lr){2-8}
& $\delta=0.8$ &
926 &
1172.250 &
1201 &
1250.100 &
1334.250 &
1734 \\
\cmidrule(lr){2-8}
& $\delta=0.6$ &
1184 &
1496 &
1561 &
1629.100 &
1810 &
2140 \\
\cmidrule(lr){2-8}
& $\delta=0.4$ &
1745 &
2174 &
2471 &
2496 &
2873 &
3087 \\
\cmidrule(lr){2-8}
& $\delta=0.2$ &
3737 &
4441.500 &
5654 &
5250.500 &
5935.500 &
6272 \\
\bottomrule
\end{tabular}
\caption{Statistics of the number of outer iterations
over the ten runs.}
\label{tab:iter}
\end{table}

}

\paragraph{\textcolor{structure}{Discussion}}{%
\begin{itemize}
\item
In terms of algorithmic computation time,
the proposed algorithm \cref{e:algo},
across all five activation fractions,
consistently outperforms the
progressive hedging algorithm \cref{e:PH}.
\item
As the activation fraction decreases,
the proposed algorithm \cref{e:algo} generally requires more outer
iterations to satisfy the stopping criterion \cref{e:crit}.
In \emph{our} implementation, we evaluate \cref{e:eval} after every
outer iteration, which requires 2,000 dense
$300\times 300$ matrix-vector multiplications,
together with projections onto the positive orthant
$C(\xi)$.
As a result, the increased number of outer iterations leads to a
greater cumulative time for monitoring the stopping criterion
\cref{e:crit}, thereby increasing the total wall-clock time.
\end{itemize}
}

\paragraph{\textcolor{structure}{Acknowledgment\autodot}}{%
I thank the two anonymous referees for constructive feedback which
helped improve the quality of the manuscript.
}

\bibliographystyle{jnsao}
\bibliography{bibli}

\end{document}